\documentclass[11pt, twoside]{article} 

\usepackage[utf8]{inputenc} 
\usepackage[margin=1in]{geometry} 
\usepackage{graphicx} 

\usepackage{booktabs} 
\usepackage{array} 
\usepackage{paralist} 
\usepackage{verbatim} 
\usepackage{mathrsfs}
\usepackage{amssymb}
\usepackage{amsthm}
\usepackage{amsmath,amsfonts,amssymb}
\usepackage{esint}
\usepackage{graphics}
\usepackage{enumerate}
\usepackage{mathtools}
\usepackage{xfrac}
\usepackage{nicefrac}
\usepackage{subcaption}
\usepackage[normalem]{ulem}
\usepackage{cancel}
\usepackage{enumitem}
\usepackage{dsfont}

\let\olddiamond\diamond
\let\oldsquare\square 

\usepackage{mathabx}

\renewcommand{\square}{\oldsquare}
\renewcommand{\diamond}{\olddiamond}

\usepackage[dvipsnames]{xcolor}
\usepackage[colorlinks=true, pdfstartview=FitV, linkcolor=blue, citecolor=blue, urlcolor=blue]{hyperref}
\usepackage[normalem]{ulem}

\usepackage{tikz}
\usetikzlibrary{calc}
\usepackage{pgf}
\usetikzlibrary{external}
\numberwithin{equation}{section}
\numberwithin{figure}{section}

\newtheorem{theorem}{Theorem}[section]

\newtheorem{proposition}[theorem]{Proposition}
\newtheorem{lemma}[theorem]{Lemma}

\theoremstyle{definition}
\newtheorem{definition}[theorem]{Definition}

\newtheorem{problem}{Open Problem}

\newcommand*{\Id}{\ensuremath{\mathrm{I}_d}}

\newcommand*{\N}{\ensuremath{\mathbb{N}}}

\newcommand*{\Z}{\ensuremath{\mathbb{Z}}}

\newcommand*{\R}{\ensuremath{\mathbb{R}}}
\newcommand*{\Zd}{\ensuremath{\mathbb{Z}^d}}
\newcommand*{\Rd}{\ensuremath{\mathbb{R}^d}}

\newcommand{\eps}{\varepsilon}

\renewcommand*{\tilde}{\widetilde}

\renewcommand{\P}{\ensuremath{\mathbb{P}}}

\newcommand{\qand}{\quad \mbox{and} \quad }

\newcommand{\sym}{\mathrm{sym}}

\newcommand{\ep}{\eps}

\newcommand{\A}{\mathcal{A}}

\DeclareMathOperator{\var}{var}

\DeclareMathSymbol{\shortminus}{\mathbin}{AMSa}{"39}
\newcommand{\nf}{\nicefrac}
\newcommand{\cs}{\mathfrak{c}_s} 
\newcommand{\css}[1]{\mathfrak{c}_{#1}}

\newcommand{\E}{\mathbb{E}}

\newcommand{\F}{\mathcal{F}}

\DeclareSymbolFont{boldoperators}{OT1}{cmr}{bx}{n}
\SetSymbolFont{boldoperators}{bold}{OT1}{cmr}{bx}{n}
\usepackage{accents}
\newcommand\thickbar[1]{\accentset{\rule{.45em}{.6pt}}{#1}}
\renewcommand{\bar}{\thickbar}
\renewcommand{\a}{\mathbf{a}}

\newcommand{\m}{\mathbf{m}}

\newcommand{\ahom}{\bar{\a}}

\newcommand{\Lams}{\gamma_1}
\newcommand{\lams}{\gamma_2}
\newcommand{\specs}{\nu}

\makeatletter 
\newcommand{\negphantom}{\v@true\h@true\negph@nt} 
\newcommand{\neghphantom}{\v@false\h@true\negph@nt} 
\newcommand{\negph@nt}{\ifmmode\expandafter\mathpalette 
	\expandafter\mathnegph@nt\else\expandafter\makenegph@nt\fi} 
\newcommand{\makenegph@nt}[1]{%
	\setbox\z@\hbox{\color@begingroup#1\color@endgroup}\finnegph@nt} 
\newcommand{\finnegph@nt}{%
	\setbox\tw@\null 
	\ifv@ \ht\tw@\ht\z@\dp\tw@\dp\z@\fi \ifh@\wd\tw@-\wd\z@\fi\box\tw@} 
\newcommand{\mathnegph@nt}[2]{%
	\setbox\z@\hbox{$\m@th #1{#2}$}\finnegph@nt} 
\makeatother

\newcommand{\Besov}[3]{\mathring{\phantom{B}}\negphantom{B}\underline{B}^{#1}_{#2,#3}}

\newcommand{\Besovnoul}[3]{\mathring{\phantom{B}}\negphantom{B}{B}^{#1}_{#2,#3}}

\newcommand{\Bhatminusul}[3]{\hat{\phantom{B}}\negphantom{B}\underline{B}^{#1}_{#2,#3}}

\newcommand{\Besovul}[3]{\underline{\phantom{B}}\negphantom{B}B^{#1}_{#2,#3}}

\def\Xint#1{\mathchoice
	{\XXint\displaystyle\textstyle{#1}}%
	{\XXint\textstyle\scriptstyle{#1}}%
	{\XXint\scriptstyle\scriptscriptstyle{#1}}%
	{\XXint\scriptscriptstyle\scriptscriptstyle{#1}}%
	\!\int}
\def\XXint#1#2#3{{\setbox0=\hbox{$#1{#2#3}{\int}$}
		\vcenter{\hbox{$#2#3$}}\kern-.5\wd0}}

\def\fint{\Xint-}

\newcommand{\avsum}{\mathop{\mathpalette\avsuminner\relax}\displaylimits}

\makeatletter
\newcommand\avsuminner[2]{%
	{\sbox0{$\m@th#1\sum$}%
		\vphantom{\usebox0}%
		\ooalign{%
			\hidewidth
			\smash{\,\rule[.23em]{8.8pt}{1.1pt} \relax}%
			\hidewidth\cr
			$\m@th#1\sum$\cr
		}%
	}%
}
\makeatother

\makeatletter
\newcommand\avsuminnerr[2]{%
	{\sbox0{$\m@th#1\sum$}%
		\vphantom{\usebox0}%
		\ooalign{%
			\hidewidth
			\smash{\,\rule[.23em]{6pt}{0.7pt} \relax}%
			\hidewidth\cr
			$\m@th#1\sum$\cr
		}%
	}%
}
\makeatother

\let\originalleft\left
\let\originalright\right
\renewcommand{\left}{\mathopen{}\mathclose\bgroup\originalleft}
\renewcommand{\right}{\aftergroup\egroup\originalright}

\newcommand{\cu}{\square}

\renewcommand{\hat}{\widehat}

\usepackage{titlesec}

\newcommand{\addperiod}[1]{#1.}
\titleformat{\section}
{\centering\normalfont\Large}{\thesection.}{0.5em}{}
\titleformat*{\subsection}{\bfseries}
\titleformat{\subsubsection}[runin]
{\normalfont\bfseries}
{\thesubsubsection.}
{0.5em}
{\addperiod}
\titleformat*{\subsubsection}{\normalfont\itshape}
\titleformat*{\paragraph}{\bfseries}
\titleformat*{\subparagraph}{\large\bfseries}

\title{A coarse-graining theory for elliptic operators and homogenization in high contrast} 

\author{Scott Armstrong
\thanks{Laboratoire Jacques-Louis Lions, Sorbonne Universit\'e 
and
Courant Institute of Mathematical Sciences, New York University.
{\footnotesize \href{mailto:scotta@cims.nyu.edu}{scotta@cims.nyu.edu}.}
}
\and
Tuomo Kuusi
\thanks{Department of Mathematics and Statistics, University of Helsinki.
{\footnotesize \href{mailto:tuomo.kuusi@helsinki.fi}{tuomo.kuusi@helsinki.fi}.}
}
}
\date{September 29, 2025}

\usepackage[nottoc,notlot,notlof]{tocbibind}
\usepackage{fancyhdr}
\usepackage{extramarks}
\begin{document}

\maketitle

\begin{abstract}
We review a coarse-graining theory for divergence-form elliptic operators. The construction centers on a pair of coarse-grained matrices defined on spatial blocks that encode a scale-dependent notion of ellipticity, transmit precise information from small to large scales, and yield coarse-grained counterparts of standard elliptic estimates. Under simplifying assumptions, we give a complete proof of the result of \cite{AK.HC} that homogenization is reached within at most $C \log^2(1+\Theta)$ dyadic length scales in the high-contrast regime, where $\Theta$ is the ellipticity contrast. We argue that this scale-local notion of ellipticity is genuinely iterable across arbitrarily many scales, providing a framework for a rigorous renormalization group analysis.
\end{abstract}


\section{Introduction}

\subsection{Coarse-graining for elliptic operators}
Many models in mathematical physics have detailed microscopic descriptions with numerous degrees of freedom. In such systems, it is natural to seek a macroscopic description via \emph{coarse-graining}: at each length scale, replace the microscopic variables by scale-dependent effective parameters constructed from block variables on spatial blocks. This procedure deliberately discards fine-scale information, with the aim of proving that, on appropriate length scales, the original system is well-approximated by the coarse-grained one with quantitative error bounds. Iterating this step---repeatedly integrating out short scales and rescaling---generates a flow of effective parameters whose fixed points encode universal large-scale behavior and yields an effective macroscopic description. This is called the \emph{renormalization group} in the physics literature, and it plays a fundamental role in quantum and statistical field theory, critical phenomena, turbulence, and other problems with many active scales.

\smallskip

Recent work of the authors develops a coarse-graining theory for divergence-form elliptic operators~$-\nabla\cdot \a(x)\nabla$. At each length scale and for each spatial block~$\cu$ we assign a pair of coarse-grained diffusion matrices, denoted by $\a(\cu)$ and $\a_*(\cu)$, respectively defined via Dirichlet and Neumann boundary-value problems on $\cu$ (for symmetric $\a$, with the definition for nonsymmetric~$\a$ more complicated). We introduce a notion of \emph{coarse-grained ellipticity} by considering the largest and smallest eigenvalues of these matrices over all sub-blocks of~$\cu$, with smaller blocks discounted rather than discarded, in the spirit of a negative-regularity (Besov-type) norm. We establish coarse-grained counterparts of many standard elliptic regularity estimates under this weaker, scale-dependent hypothesis. From the probabilistic point of view, we prove two-sided, scale-dependent bounds on the variance growth (effective diffusivity) of the associated diffusion process, quantitatively linked to the coarse-grained ellipticity constants.

\smallskip

Analytically, the key advantage of working with the coarse-grained matrices is \emph{closure}: the coarse-grained data on a larger block are controlled in terms of the same kind of data on smaller blocks, enabling a genuine scale-by-scale induction. In particular, the family~$(\a(\cu),\a_*(\cu))$, indexed over all blocks up to a given scale, retains enough information about the behavior of the solutions to recover the coarse-grained versions of the essential elliptic estimates we require. Our scheme assumes coarse-grained ellipticity at a given scale and, under appropriate mixing and integrability hypotheses, \emph{improves} it at the next scale via an explicit update inequality; because the hypothesis and conclusion live in the same class, the step iterates across arbitrarily many scales. This iterability is what makes the construction a rigorous renormalization group scheme.

\smallskip

Our approach grows out of the quantitative homogenization program developed jointly with Mourrat and Smart~\cite{AS,AKM1,AKMbook}, where the coarse-grained matrices and multiscale iterations already appear. 
However, as with the quantitative theory of Gloria, Neukamm and Otto~\cite{GNO,GNO2}, the quantitative estimates driving these iterations depend on the microscopic uniform ellipticity ratio and therefore are not effective in regimes with many active scales or when the uniform ellipticity ratio is large. The idea of viewing homogenization as coarse-graining or of defining a ``box diffusivity'' via variational energies in cubes is classical (see~\cite{AvM,MK,Fann}) and the idea of using the coarse-grained quantity~$\a(\cu)$ to prove homogenization goes back to De Giorgi and Spagnolo~\cite{DGS,DeG}.  Homogenization has previously been linked to rigorous renormalization group arguments in~\cite{BK,BKL,SZ} and is directly used in a deterministic context in~\cite{AV}. What is new in our approach is the closure of the analytic framework at the coarse-grained level. 

\smallskip

In this article, we review some basic elements of the coarse-graining theory. To illustrate its efficacy as a renormalization scheme, we give a complete proof in a streamlined setting (symmetric fields, isotropy in law, finite-range dependence) of the recent high-contrast result from~\cite{AK.HC} that homogenization is reached within at most~$C\log^2(1+\Theta)$ dyadic scales, where~$\Theta$ is the coarse-grained ellipticity contrast. 
While not discussed here, we also mention our companion work, joint with Bou-Rabee~\cite{ABK.SD}, where the same coarse-grained machinery is used to prove a (quenched) \emph{superdiffusive central limit theorem} for a particle in a critically correlated random drift. Both examples sit at or near criticality in the sense that the analysis must range across arbitrarily many scales, and together they illustrate how coarse-grained ellipticity supplies a framework under which the iterations underlying these arguments can be closed.

\subsection{High contrast homogenization} 

Consider a coefficient field~$\a(x)$ valued in~$\R^{d \times d}$ and satisfying, for some constants~$0<\lambda\leq \Lambda <\infty$, the uniform ellipticity condition
\begin{equation}
\label{e.UE}
\lambda |e|^2 \leq e \cdot \a (x)e
\qquad \mbox{and} \qquad 
\Lambda^{-1} |e|^2 \leq e \cdot \a^{-1} (x)e \,,
\quad \forall 
x,e\in\Rd
\,.
\end{equation}
If we assume also that~$\a(x)$ is a stationary random field satisfying a quantitative mixing condition (such as finite range dependence), then with high probability the operator~$-\nabla \cdot \a(x)\nabla$ will be quantitatively close, on large scales, to a deterministic, constant-coefficient operator~$-\nabla \cdot \ahom \nabla$. This principle is referred to as \emph{homogenization}. 
The question which interests us is what ``on large scales'' precisely means. Starting at the length scales of the correlations of the field, how far must we zoom out before we see homogenization? The aim of the \emph{high-contrast homogenization problem} is to quantify the minimal scale at which homogenization appears---the \emph{homogenization length scale}---as a function of the ellipticity contrast~$\Theta\coloneqq \Lambda/\lambda$. 

\smallskip

Over the past fifteen years, quantitative stochastic homogenization has produced sharp and explicit error estimates in the regime of large scale separation~\cite{GO1,GNO,AS,GNO2,AKM1,AKMbook}. However, the constants in these bounds deteriorate badly with the ellipticity contrast~$\Theta$ is large, so the estimate of the minimal scale needed to reach a fixed error tolerance can be extremely large when~$\Theta \gg 1$. 
While not explicitly estimated (since all constants were allowed to depend on~$\Theta$), the best bound obtained by chasing the constants through these papers is
\begin{equation*}
\mbox{homogenization length scale} 
\ \lesssim \ 
C \exp \bigl( C \Theta^{\nf12} \log \Theta \bigr) \,.
\end{equation*}
Here~$C<\infty$ is a constant depending only on~$d$ and the quantitative mixing condition.

\smallskip 

A much better estimate was obtained recently by the authors~\cite{AK.HC}, where we also consider fields satisfying a much weaker, coarse-grained ellipticity condition (see Definition~\ref{def.cg.ellipticity} below). 

\begin{theorem}[\cite{AK.HC}]
\label{t.AK.HC}
The homogenization length scale is at most~$C \exp ( \log^2 \Theta )$. 
\end{theorem} 

See~\cite{AK.HC} for a more precise statement of the theorem. A complete, self-contained proof of a special case of Theorem~\ref{t.AK.HC} is presented in Section~\ref{s.HC}. 

\smallskip

The bound in Theorem~\ref{t.AK.HC} is not expected to be sharp, and it is an important open problem to obtain an estimate which is a polynomial in~$\Theta$. 

\begin{problem}[Polyomial upper bound]
In the case the coefficient field~$\a(\cdot)$ satisfies a finite range of dependence, prove that the homogenization length scale is at most~$C \Theta^{C}$ for some~$C(d)<\infty$. 
\end{problem}

At a high level, the proof of Theorem~\ref{t.AK.HC} combines a scale-local closure at the coarse-grained level with a selection of good scales (via a dyadic pigeonholing argument) and the properties of the coarse-grained diffusion matrices to show that the coarse-grained ellipticity contrast must contract by a factor of~$\nf12$ after~$C \log \Theta$ many scales. We then iterate this result~$C \log \Theta$ times to obtain the theorem. See below in Section~\ref{s.HC} for more details. 

\smallskip

Since the 1970s, what we call the high-contrast homogenization problem has been studied in the physics literature through \emph{random resistor networks}. Equip the lattice~$\Zd$ with nearest-neighbor edges, and independently assign each edge one of two finite conductances~$\sigma_{\mathrm{m}}$ and~$\sigma_{\mathrm{d}}$ with~$\sigma_{\mathrm{m}}\gg \sigma_{\mathrm{d}}$. Denote the contrast by~$h\coloneqq \sigma_{\mathrm{d}}/\sigma_{\mathrm{m}}$, let~$p$ be the probability of an edge having conductance~$\sigma_{\mathrm{m}}$. We are interested in the case in which~$h$ is very small and~$p$ is tuned near the percolation threshold~$p_c$. A key insight, emphasized in the review~\cite{clerc1990binary} and already present in early works \cite{efros1976critical,bergman1977dielectric,straley1976critical,straley1977position,straley1977critical}, is that bond percolation and the resistor network can be treated jointly as a two-variable critical system with relevant coordinates~$(p-p_c,h)$. In this view, the point~$h=0$ and~$p=p_c$ is an isolated critical point with two relevant directions, exactly as at the Curie point of a ferromagnet where temperature~$T-T_c$ and magnetic field~$H$ are the natural variables (the same letter~$h$ is used in the resistor-network literature to stress this analogy, see~\cite{clerc1990binary}). 
We stress that this discrete model is directly analogous to our continuum setup---and our arguments can be adapted to the discrete setting with essentially minor notational changes. 
Note that~$h$ is related to the parameter~$\Theta$ we call the ellipticity contrast via~$\Theta = h^{-1}$. 

\smallskip

This two-variable perspective leads to a scaling ansatz for the effective conductivity,
\begin{equation}
\sigma_{\mathrm{eff}} \asymp
\sigma_{\mathrm{m}}
|p-p_c|^{t}
\Phi_{\pm}
\bigl( h |p-p_c|^{-(s+t)} \bigr)
\,,\qquad p > p_c\,,
\label{e.physics.prediction}
\end{equation}
with an analogous form for the range~$p<p_c$, where~$s$ and~$t$ are the conductivity exponents respectively for~$p<p_c$ and~$p>p_c$,  and~$\Phi_{\pm}$ are universal scaling functions depending only on dimension~\cite{straley1976critical,clerc1990binary}. The ansatz~\eqref{e.physics.prediction} predicts a regime~$h\lesssim |p-p_c|^{s+t}$ in which the smallness of~$h$ dictates the behavior (this is our high-contrast homogenization problem) and another regime~$h\gtrsim |p-p_c|^{s+t}$ in which it is instead the geometry of the percolation cluster that controls the effective diffusivity. The term ``crossover'' refers to the switch between these two regimes as~$h|p-p_c|^{-(s+t)}$ passes through order one \cite{clerc1990binary,harris1987diluted,harris1987potts,stenull1999critical}. The exponents~$s$ and~$t$ admit definite numerical predictions in low dimensions and mean-field values in high dimension; see \cite{straley1977critical,clerc1990binary} for representative estimates and discussion.

\smallskip

Combining the crossover condition~$h|p-p_c|^{-(s+t)}\asymp 1$ with the predicted scaling (see~\cite{kirkpatrick1973percolation}) of the percolation correlation length~$\xi$ near the critical point, 
\begin{equation}
\xi\asymp |p-p_c|^{-\nu}
\,,
\label{e.percolation.correlation}
\end{equation}
yields a length scale
\begin{equation}
\xi_h \asymp h^{-\nu/(s+t)}
\,,
\label{e.homogenization.crossover}
\end{equation}
which is the physics prediction for the homogenization length scale for~$h>0$ and~$p=p_c$, see~\cite{clerc1990binary,harris1987diluted,harris1987potts,stenull1999critical}. For general~$p$, the macroscopic length controlling effective behavior is $\min\{\xi,\xi_h\}$. 
Note that~$\xi_h$ is predicted to diverge as~$h\to 0$ in all dimensions. For~$d>6$, the exponent~$\nu/(s+t)$ is predicted to be equal to~$\nf16$. Numerical evidence suggests a value of approximately~$0.513$ in~$d=2$ and~$0.33$ in~$d=3$; see~\cite{clerc1990binary}.

\smallskip

In the continuum setting, our high-contrast result yields a quantitative upper bound on the homogenization length in terms of contrast alone. Translating to the resistor-network parameter~$h=\Theta^{-1}$, we obtain
\begin{equation}
\xi_h \lesssim \exp\bigl(C\left|\log h\right|^2\bigr)
\,.
\label{e.AK.HC.xi.h}
\end{equation}
For comparison, the best general rigorous bounds for the percolation correlation length on $\mathbb{Z}^d$ currently give (see \cite{DKT}) the exponential-type bound
\begin{equation}
\xi \lesssim \exp\bigl(C|p-p_c|^{-2}\bigr)
\,.
\label{e.DKT}
\end{equation}
It is natural to wonder whether an estimate like~\eqref{e.AK.HC.xi.h} can be used to get information about~$\xi$. 
Heuristically, the crossover relation~$h|p-p_c|^{-(s+t)}\asymp 1$ links~$\xi_h$ and~$\xi$; thus, a quantitative crossover theorem (even in a suboptimal form) could, in principle, transmit estimates of the type \eqref{e.AK.HC.xi.h} into near-critical percolation. Establishing such a bridge is an interesting open problem.

\begin{problem}
Establish a quantitative version of one side of the predicted crossover relation by showing that, for~$h \lesssim f( |p-p_c| )$ for some reasonable (hopefully powerlike) function~$f$, the macroscopic length controlling effective behavior of the random resistor network is~$\xi_h$.
\end{problem}

\section{Coarse-graining theory}

In any block-averaging scheme one must decide what, exactly, is being averaged. For a divergence-form elliptic operator, the coefficient field~$\a(x)$ is not itself an observable quantity, as ``diffusivity'' is defined only through the constitutive relation between flux and gradient and must therefore be inferred from the response of the medium to imposed loads. Averaging~$\a(x)$ directly is therefore not the natural macroscopic operation. What is meaningful is the \emph{response map} that sends block-averaged gradients to the corresponding block-averaged flux. 

\smallskip

Ideally, for each block~$\cu \subseteq\Rd$ we would like to find a matrix~$\a(\cu)$ satisfying 
\begin{equation}
\a(\cu) \fint_{\cu} \nabla u(x) \,dx 
=
\fint_{\cu} \a(x) \nabla u(x) \,dx\,,
\qquad 
\forall u \in \mathcal{A}(\cu)\,,
\label{e.too.much.to.hope.for}
\end{equation}
where~$\mathcal{A}(\cu)$ is denotes the linear space of~$\a$-harmonic functions in~$\cu$. 
Clearly we cannot expect a single~$d$-by-$d$ matrix~$\a(\cu)$ to satisfy~\eqref{e.too.much.to.hope.for} for every~$\a$-harmonic function~$u$.
Instead, we isolate two natural~$d$-dimensional subspaces of~$\mathcal{A}(\cu)$---the solutions of affine Dirichlet and Neumann problems---and define block matrices~$\a(\cu)$ and~$\a_*(\cu)$, respectively, so that the identity~\eqref{e.too.much.to.hope.for} holds on these subspaces. 

\smallskip

This pair of matrices turns out to be very convenient for building a coarse-graining theory. First,~$\a(\cu)$ and~$\a_*^{-1}(\cu)$ are \emph{subadditive} and  \emph{local} (they depend only on~$\a\vert_\cu$), which makes it convenient to apply mixing conditions.  
Second, the gap~$\a(\cu) - \a_*(\cu)$ is nonnegative and controls the defect of the block response for \emph{arbitrary}~$\a$-harmonic function~$u$, providing quantitative error bounds.
Finally, the pair carries enough quantitative information to define a scale-dependent notion of ellipticity (coarse-grained ellipticity), which will substitute for uniform ellipticity in large-scale estimates.

\smallskip

In this section we will review the definitions and basic properties of the coarse-grained matrices, the notion of coarse-grained ellipticity and explore some generalizations of elliptic and functional inequalities. The presentation here is an abbreviated version of what appears in~\cite{AK.HC} or~\cite{AK.Book}, and these can be consulted for further discussion and the details of some arguments. 
We stress that the presentation in this section is purely deterministic: no randomness of the field is needed to define the coarse-grained objects. 

\smallskip

To keep our presentation simple, we restrict our presentation here to the case of coefficient fields~$\a(x)$ valued in the \emph{symmetric} matrices. We stress that this restriction is merely a pedagogical choice. The coarse-graining theory developed in~\cite{AK.HC}, including the proof of Theorem~\ref{t.AK.HC}, applies to general (not necessarily symmetric) fields.

\subsection{Notation for functional spaces}

We list some basic notation used throughout the article. Any notation which is unclear to the reader and not mentioned here can be found in~\cite[Section 1.5]{AK.HC}. 
For each open~$U\subseteq\Rd$,  we let~$\Omega(U)$ denote the set of possible coefficient fields, defined by
\begin{equation}
\Omega(U)
\coloneqq 
\bigl\{ \a : U \to \R^{d\times d}_{\sym,+} \,:\,
\a,\, \a^{-1} 
\in L^1_{\mathrm{loc}}(U;\R^{d\times d})
\bigr\}
\,.
\label{e.Omega.U.def}
\end{equation}
Here~$\R^{d\times d}_{\sym,+}$ is the set of~$d$-by-$d$ matrices with real entries which are symmetric and positive definite. For each~$U \subseteq\Rd$ and~$\a\in \Omega(U)$, we define the function spaces~$H^1_\a(U)$ as the completion of~$C^\infty(U)$ with respect to the norm
\begin{equation}
\label{e.H1a}
\big\| u \bigr\|_{H^1_\a(U)} \coloneqq \Bigl( \| u \|_{L^1(U)}^2 + \int_U \nabla u \cdot \a \nabla u\Bigr)^{\nicefrac12} \,.
\end{equation}
Observe that, by H\"older's inequality, we have 
\begin{equation}
\label{e.Omega.imp}
\a \in \Omega(U),  \ u\in H^1_\a(U)
\implies 
\nabla u , \,\a\nabla u \in L^1_{\mathrm{loc}}(U)
\,.
\end{equation}
The space~$H^1_\a(U)$ is a complete Hilbert space for every open~$U\subseteq\Rd$ and~$\a\in\Omega(U)$. It is clear that~$C^\infty_c (U) \subseteq H^1_{\a}(U)$. We also define~$H^1_{\a,0}(U)$ as the closure of~$C^\infty_c (U)$ with respect to~$\| \cdot \|_{H^1_\a(U)}$.
For each bounded domain~$U \subseteq\Rd$, we let~$\mathcal{A}(U)$ denote the set of~$\a$-harmonic functions in~$U$, 
\begin{equation*}
\mathcal{A}(U)
\coloneqq 
\bigl\{ u \in H^1_\a(U)\,:\, \nabla \cdot \a\nabla u = 0 \bigr\}
\,.
\end{equation*}
We denote affine functions by~$\ell_p(x)\coloneqq p\cdot x$.

We denote averages using a slash through the sum symbol; if~$S$ is a finite set, then 
\begin{equation*}
\avsum_{i\in S} a_i 
\coloneqq
\frac{1}{|S|} \sum_{i\in S} a_i
\,.
\end{equation*}
We work with triadic cubes: for each~$n\in\Z$, we define
\begin{equation*}
\cu_n \coloneqq \Bigl( -\frac12 3^n , \frac12 3^n \Bigr)^{d} 
\,.
\end{equation*}
For every~$m,n\in\Z$ with~$m>n$ and~$z\in 3^m\Zd$, the collection~$\{ z' + \cu_n \,:\, z' \in 3^n\Zd \cap (z+\cu_m) \}$ is a partition (up to a Lebesgue null set) of the cube~$z+\cu_m$.  

We use slashes and underlines to denote volume-normalization. 
Volume-normalized integrals and~$L^p$ norms are denoted, for~$p\in[1,\infty)$, by
\begin{equation}
\label{e.volume.normalize}
(f)_U  \coloneqq  
\fint_U f(x) \,dx  \coloneqq  \frac{1}{|U|} \int_U f(x)\,dx
\qquad \mbox{and} \qquad 
\| f \|_{\underline{L}^p(U)} \coloneqq  \Bigl( \fint_U |f(x)|^p \,dx \Bigr)^{\nicefrac1p}
\,.
\end{equation}
We next introduce our notation for Besov spaces. 
For each~$s \in (0,1)$, $p\in [1,\infty)$, $q\in (0,\infty)$ and~$n\in \N$, we define a (volume-normalized) Besov seminorm in the cube~$\cu_n$ by
\begin{equation}
\label{e.Bs.seminorm}
\left[ g \right]_{\underline{B}_{p,q}^{s}(\cu_{n})}
 \coloneqq 
\Biggl( 
\sum_{k=-\infty}^n
3^{- s q k} \biggl( 
\avsum_{z\in 3^{k-1}\Zd, \, z + \cu_k \subseteq \cu_n}
\bigl \| g - (g)_{z+\cu_k}\bigr \|_{\underline{L}^p(z+\cu_k)}^p
\biggr)^{\!\nicefrac qp}
\Biggr)^{\! \nicefrac1q}
\,.
\end{equation}
Extensions to~$p=\infty$ or~$q=\infty$ are defined in the usual way. 
The dual seminorms of negative regularity are defined as follows. For every~$s \in (0,1]$,~$p\in [1,\infty]$ and~$q\in [1,\infty]$, we 
let $p'$ and~$q'$ denote the H\"older conjugate exponents of~$p$ and~$q$, respectively, and we define
\begin{equation}
\label{e.Bs.minus.seminorm}
\left[ f \right]_{\Bhatminusul{-s}{p}{q} (\cu_{n})}
 \coloneqq 
\sup \biggl\{ \fint_{\cu_n} f g \, : \, g \in B_{p',q'}^{s}(\cu_{n}) \,, \; 
\left\|  g \right\|_{\underline{B}_{p',q'}^{s}(\cu_{n})} \leq 1 \biggr\}
\,.
\end{equation}
The dual space of the subspace of~$B_{p',q'}^{s}(\cu_n)$ with zero boundary values is defined by 
\begin{equation}
\label{e.Bs.minus.seminorm.zero}
\left[ f \right]_{\underline{B}_{p,q}^{-s}(\cu_{n})}
 \coloneqq 
\sup \biggl\{ \fint_{\cu_n} f g \, : \, g \in C_{\mathrm{c}}^\infty(\cu_n) \,, \; 
\left[  g \right]_{\underline{B}_{p',q'}^{s}(\cu_{n})} \leq 1 \biggr\}
\,.
\end{equation}
Finally, we introduce another variant of these negative spaces by defining
\begin{equation}
\label{e.Bs.minus.seminorm.explicit}
\left[ f \right]_{\Besov{-s}{p}{q}(\cu_n)}
 \coloneqq 
\biggl(
\sum_{k=-\infty}^n
3^{s qk}
\biggl(
\avsum_{z\in 3^k\Zd \cap \cu_n}
\bigl| (f)_{z+\cu_k}\bigr |^p
\biggr)^{\!\nicefrac qp}
\biggr)^{\! \nicefrac1q}
\,.
\end{equation}
We will write~$f \in \Besovnoul{-s}{p}{q}(\cu_n)$ if the quantity in~\eqref{e.Bs.minus.seminorm.explicit} is finite. Note that~$f\geq 0$ and~$f \in \Besovnoul{-s}{p}{q}(\cu_n)$ implies~$f\in L^1(\cu_n)$. 
The definition~\eqref{e.Bs.minus.seminorm.explicit} will be useful when estimating the negative seminorms since, for every~$f$, 
\begin{equation}
\label{e.weak.norms.ordering}
\left[ f \right]_{\underline{B}_{p,q}^{-s}(\cu_{n})}
\leq
\left[ f \right]_{\Bhatminusul{-s}{p}{q} (\cu_{n})}
\leq
3^{d+s}
\left[ f \right]_{\Besov{-s}{p}{q}(\cu_n)}
\,.
\end{equation}
See~\cite{AK.HC}. In particular, we have the ``duality pairing'' 
\begin{equation}
\label{e.duality.Bs}
\biggl| \fint_{\cu_m} f g \biggr| 
\leq
C
\| g  \|_{\Besovul{s}{p'}{q'}(\cu_m)}
 \left[ f \right]_{\Besov{-s}{p}{q}(\cu_{m})}
\,.
\end{equation}
We will need two other straightforward facts which are proved in the appendix of~\cite{AK.HC}. First, there exists~$C(d)<\infty$ such that, for every~$m \in \Z$,~$s \in [0,1)$ and~$u \in B^{s}_{2,\infty}(\cu_{m})$,
\begin{equation}  
\label{e.divcurl.est0}
\left\| u- (u)_{\cu_m}  \right\|_{\underline{B}_{2,\infty}^{s}(\cu_m)}
\leq 
C
\left[ \nabla u  \right]_{\Besov{s-1}{2}{1}(\cu_{m})}
\,.
\end{equation}
More generally, if~$\varphi \in W^{2,\infty}(\cu_m)$, then
\begin{equation}
\label{e.divcurl.est1}
\left\| (u- (u)_{\cu_m}) \nabla \varphi \right\|_{\underline{B}_{2,\infty}^{s}(\cu_m)}
\leq 
C
3^{m}
 \| \nabla \varphi\|_{\underline{W}^{1,\infty}(\cu_m)}
\left[ \nabla u  \right]_{\Besov{s-1}{2}{1}(\cu_{m})}
\,.
\end{equation}

\subsection{Coarse-grained matrices}

The well-posedness of boundary-value problems in bounded domains~$U$ for arbitrary~$\a\in \Omega(U)$ follows from Lax-Milgram; see \cite[Section 2.1]{AK.HC}. With this in mind, we define, for each bounded domain~$U \subseteq\Rd$ and~$\a \in \Omega(U)$, the \emph{coarse-grained matrices}~$\a(U)$ and~$\a^{-1}_*(U)$ by 
\begin{equation}
\left\{
\begin{aligned}
& 
\frac12 p \cdot \a(U) p 
\coloneqq 
\min_{w \in \ell_p + H^1_{\a,0}(U)} 
\fint_U 
\frac12 \nabla w \cdot \a\nabla w \,, \quad \forall p\in\Rd\,, \\
& 
\frac12 q \cdot \a_*^{-1} (U) q 
\coloneqq 
\max_{w \in H^1_{\a}(U)} 
\fint_U 
\Bigl( q \cdot \nabla w 
-
\frac12 \nabla w \cdot \a\nabla w 
\Bigr) \,,  \quad \forall q\in\Rd\,. 
\end{aligned}
\right.
\label{e.a.astar.def}
\end{equation}
The quantity~$\a(U)$ is classical in the theory of homogenization. In fact, it was introduced in one of the first mathematical papers on homogenization in the work of De Giorgi and Spagnolo~\cite{DGS}. 
To  our knowledge, the second quantity was not put to use until the work of the first author and Smart~\cite{AS} where it was introduced, in the context of quantitative homogenization, and used to implement (what we would now call) a coarse-graining scheme. 

\smallskip

There is a convenient variational formulation that combines both of the coarse-grained matrices in a single expression: we define, for each~$p,q\in\Rd$,  
\begin{equation}
J(U,p,q) 
\coloneqq 
\max_{w \in \A(U)} 
\fint_{U} 
\Bigl( 
-\frac12 \nabla w \cdot \a \nabla w 
+ q \cdot \nabla w - p\cdot \a\nabla w \Bigr) 
\,.
\label{e.J.def}
\end{equation}
Note that the maximum in~\eqref{e.J.def} is over the set of~$\a$-harmonic functions in $U$. 
By performing some computations, one finds that
\begin{equation}
J(U,p,q) 
= 
\frac12 p \cdot \a(U) p 
+
\frac12 q \cdot \a_*^{-1} (U) q 
-p\cdot q
\,.
\label{e.J.a.astar}
\end{equation}
The maximizer~$w$ in~\eqref{e.J.def} is unique up to additive constants, and we denote it by~$v(\cdot,U,p,q)$. By straightforward manipulations, we find that~$(p,q) \mapsto  v(\cdot,U,p,q)$ is linear, that~$v(\cdot,U,p,0)$ is the solution of the Dirichlet problem with affine data~$\ell_{-p}$ and~$v(\cdot,U,0,q)$ is the solution of the Neumann problem with prescribed boundary flux~$\mathbf{n}\cdot q$. 

\smallskip

The variational definition of the quantity~$J$ makes it subadditive by definition (test the large scale maximizer in the definitions of the subcube partition). In particular, for every~$m,n\in\Z$ with~$n<m$, 
\begin{equation}
J(\cu_m,p,q) 
\leq 
\avsum_{z\in 3^n\Zd\cap \cu_m} 
J(z+\cu_n,p,q)\,.
\label{e.subadditivity}
\end{equation}
It is immediate that~$J(U,p,q)$ is local in the sense that it depends on the field~$\a(x)$ only via~$\a\vert_U$. 

\smallskip

It is clear from its definition that~$J$ is nonnegative (test with the zero function), and we deduce from this and~\eqref{e.J.a.astar} the ordering
\begin{equation}
\a_*(U) \leq \a(U)
\,.
\label{e.ordering}
\end{equation}
By testing~\eqref{e.a.astar.def} with affine functions, we find that the coarse-grained quantities also satisfy the integral bounds
\begin{equation}
\a(U) 
\leq 
\fint_U \a(x)\,dx 
\qquad \mbox{and} \qquad
\a_*^{-1}(U) 
\leq 
\fint_U \a^{-1}(x)\,dx 
\,.
\label{e.cg.integral.bounds}
\end{equation}
By computing the first variation of~\eqref{e.J.def}, we obtain the identity
\begin{equation}
q \cdot \fint_U \nabla w 
- p \cdot \fint_U \a \nabla w 
=
\fint_U 
\nabla w 
\cdot \a \nabla v(\cdot,U,p,q)
\,, \quad \forall w \in \A(U)\,.
\label{e.J.first.var}
\end{equation}
Hence, for every~$w \in \A(U)$, 
\begin{equation}
J(U,p,q) - \fint_U \frac12 \nabla w \cdot \a\nabla w 
=
\fint_U \frac12 (\nabla v(\cdot,U,p,q) - \nabla w)  \cdot \a (\nabla v(\cdot,U,p,q) - \nabla w)
\,.
\label{e.J.second.var}
\end{equation}
In particular, 
\begin{equation}
J(U,p,q) 
=
\fint_U \frac12 \nabla v(\cdot,U,p,q) \cdot \a \nabla v(\cdot,U,p,q)
\,.
\label{e.J.energy.of.v}
\end{equation}
As promised in our discussion above around~\eqref{e.too.much.to.hope.for}, one can obtain from these identities that
\begin{equation}
\left\{
\begin{aligned}
& 
\fint_{U} \a\nabla v(\cdot,U,p,0) = \a(U) \fint_U \nabla v(\cdot,U,p,0)  \,, 
\\
& 
\fint_{U} \a\nabla v(\cdot,U,0,q) = \a_*(U) \fint_U \nabla v(\cdot,U,0,q) \,.
\end{aligned}
\right.
\label{e.exact.response}
\end{equation}
We next show how~$\a(U)$ and~$\a_*(U)$ give us information about arbitrary~$\a$-harmonic functions. 
According to~\eqref{e.J.first.var},~\eqref{e.J.energy.of.v} and Cauchy-Schwarz, for any~$w\in\A(U)$ and~$p,q\in\Rd$, 
\begin{equation}
\label{e.fluxmaps}
\biggl | \fint_{U} \bigl ( p \cdot \a \nabla w - q \cdot \nabla w \bigr ) \biggr |
=
\biggl | \fint_U \nabla w 
\cdot  \a \nabla v\bigl (\cdot, U, p,q \bigr )  \biggr |
\leq
(2J \bigl (U, p,q \bigr ) )^{\nicefrac12}
\Bigl( \fint_U \nabla w \cdot \a \nabla w \Bigr)^{\!\nicefrac12}
\,.
\end{equation}
Plugging in~$q = \a_*(U)p$, and observing from~\eqref{e.J.a.astar} that
\begin{equation*}
J \bigl (U, p,\a_*(U)p \bigr ) = \frac12 p\cdot (\a(U) - \a_*(U) ) p 
\,,
\end{equation*}
we obtain
\begin{equation*}
\biggl| p\cdot  \Bigl( \fint_{U} \a \nabla w - \a_*(U) \fint_{U} \nabla w \Bigr ) \biggr|
\leq
(p\cdot (\a(U) - \a_*(U) ) p  )^{\nicefrac12}
\Bigl( \fint_U \nabla w \cdot \a \nabla w \Bigr)^{\!\nicefrac12}
\,.
\end{equation*}
Taking the maximum over~$|p|=1$ yields 
\begin{equation}
\biggl|\fint_{U} \a \nabla w - \a_*(U) \fint_{U} \nabla w \biggr|
\leq
\bigl| \a(U) - \a_*(U) \bigr|^{\nicefrac12}
\Bigl( \fint_U \nabla w \cdot \a \nabla w \Bigr)^{\!\nicefrac12}
\,,
\label{e.response.map}
\end{equation}
where~$|A|$ denotes the spectral norm of a matrix. 
This simple inequality~\eqref{e.response.map} is a very powerful tool in the coarse-graining theory. 

\smallskip

The identities~\eqref{e.exact.response} say that the block response is exact for~$\a(U)$ and~$\a_*(U)$, respectively, on the subspaces~$\{ v(\cdot,U,p,0) \,:\, p\in\Rd\}$ and~$\{ v(\cdot,U,0,q) \,:\, q\in\Rd\}$; meanwhile, the inequality~\eqref{e.response.map} says that the block responses are stable elsewhere, with an error proportional to (the square root of) the gap~$\a(U) - \a_*(U)$. This gap, which we can consider as a \emph{coarse-graining defect}, is therefore a primary quantity of interest.   

\smallskip

We also have the following key inequalities, which demonstrate that the energy of an~$\a$-harmonic function controls the block energy of its block gradient and flux:
\begin{equation}
\label{e.energymaps}
\frac12\Bigl( \fint_U \nabla w \Bigr) \cdot \a_*(U) \Bigl( \fint_U \nabla w \Bigr)
\leq
\fint_U \frac12 \nabla w \cdot \a\nabla w 
\,, \quad 
\forall w\in H^1_{\a} (U)\,. 
\end{equation}
Similarly, the coarse-grained matrix~$\a(U)$ gives a lower bound for the spatial average of the flux of an arbitrary solution in terms of its energy: 
\begin{align}
\label{e.energymaps.flux}
\frac12\Bigl( \fint_U \a \nabla w \Bigr) \cdot \a^{-1} (U) \Bigl( \fint_U \a \nabla w \Bigr)
\leq
\fint_U \frac12 \nabla w \cdot \a\nabla w 
\,, \quad 
\forall w\in \mathcal{A}(U)\,.
\end{align}

\subsection{Coarse-grained ellipticity}

The block matrices above describe response on a single region via block averages, but macroscopic control requires testing against smooth functions across all scales; that is, we need quantitative information in weak (negative-regularity) topologies. We therefore define several multiscale, auxiliary quantities: for each triadic cube~$\cu$ we record the matrices~$\a(\cu)$,~$\a_*^{-1}(\cu)$ as well as the gap~$\a(\cu)-\a_*(\cu)$ and the deviation~$\a(\cu)-\ahom$, and we aggregate these over all scales with finer scales discounted. The first pair defines coarse-grained ellipticity constants that play the roles of~$\Lambda$ and~$\lambda^{-1}$ at a given observation scale; the gap and the deviation quantify, at that scale, the defect of the block response and the distance to the effective operator~$-\nabla\cdot\ahom\nabla$, respectively. These local, iterable parameters are the basis for coarse-grained elliptic and functional inequalities that drive a scale-by-scale analysis of~$-\nabla\cdot\a\nabla$.

\begin{definition}[Coarse-grained ellipticity constants]
\label{def.cg.ellipticity}
For every~$s,t\in (0,\infty)$,~$q\in [1,\infty)$,~$m\in\Z$ and coefficient field~$\a \in \Omega(\cu_m)$, we set~$\css{u} := 1-3^{-u}$ and define the multiscale composite quantities
\begin{equation}
\label{e.coarse.grained.ellipticity}
\left\{
\begin{aligned}
& {\Lambda}_{s,q}(\cu_m)
 \coloneqq  
\biggl( 
\css{sq} \sum_{k=-\infty}^{m} 
3^{-sq(m-k)} 
\max_{z\in 3^k\Zd \cap \cu_m} 
\bigl| \a(z+\cu_k) \bigr|^{\nf q2} 
\biggr)^{\!\nf{2}{q}}
\,, \\  &
{\lambda}_{t,q}(\cu_m) 
 \coloneqq 
\biggl(\css{tq} \sum_{k=-\infty}^{m} 
3^{-tq(m-k)} 
\max_{z\in 3^k\Zd \cap \cu_m} 
\bigl| \a_{*}^{-1}(z+\cu_k) \bigr|^{\nf q2}
\biggl)^{\!- \nf{2}{q}}
\,.
\end{aligned}
\right.
\end{equation}
For~$q=\infty$, the sum above is replaced by the supremum over the scale parameter~$k$ and~$\css{\infty} =1$.
We say that a field~$\a \in\Omega(\cu_m) $ is \emph{coarse-grained elliptic in~$\cu_m$} if there exists~$s,t<1$ with~$s+t<1$ such that~$\Lambda_{s,1}(\cu_m)<\infty$ and~$\lambda_{t,1}(\cu_m)>0$.
\end{definition}

Note that we have the trivial ordering~$\Lambda_{s,\infty}(\cu_m) \leq  \Lambda_{s,1}(\cu_m)$ since~$k \mapsto \max_{z\in 3^k\Zd \cap \cu_m} 
\bigl| \a(z+\cu_k) \bigr|$ is nonincreasing due to subadditivity. Similarly,~$\lambda_{t,\infty}^{-1}(\cu_m) \leq \lambda_{t,1}^{-1}(\cu_m)$. 

\smallskip

The definitions in~\eqref{e.coarse.grained.ellipticity} can be compared to the negative regularity Besov-type norm in~\eqref{e.Bs.minus.seminorm.explicit}. In fact, it is immediate from the integral bounds in~\eqref{e.cg.integral.bounds} and the definitions~\eqref{e.coarse.grained.ellipticity} that 
\begin{equation*}
\Lambda_{s,q}(\cu_m) 
\leq 
\css{sq}^{\nf2q} 3^{-2sm} 
[ \a ]_{\Besov{-2s}{\infty}{\nf q2}(\cu_m)} 
\quad \mbox{and} \quad 
\lambda_{t,q}^{-1} (\cu_m) 
\leq 
\css{tq}^{\nf2q} 3^{-2tm} 
[ \a^{-1} ]_{\Besov{-2t}{\infty}{\nf q2}(\cu_m)} 
\,.
\end{equation*}
In particular, if~$s+t<1$, then we have the implication
\begin{equation}
\a \in \Besovnoul{-2s}{\infty}{\nf 12}(\cu_m) 
\ \mbox{and} \ 
\a^{-1} \in \Besovnoul{-2t}{\infty}{\nf 12} (\cu_m) 
\  \implies \ \mbox{$\a(\cdot)$ is coarse-grained elliptic in~$\cu_m$.}
\label{e.negative.reg.imp}
\end{equation}
Coarse-grained ellipticity should be seen as a \emph{negative regularity} type condition on the coefficient field and its inverse. 

\smallskip

If~$\a$ satisfies the uniform ellipticity condition~\eqref{e.UE}, then~$H^1_\a(\cu_m) = H^1(\cu_m)$ and the Poincar\'e inequality can be written in terms of the~$H^1_\a(\cu_m)$ seminorm as 
\begin{equation}
\| u - (u)_{\cu_m} \|_{\underline{L}^2(\cu_m)}^2 
\leq 
C \lambda^{-1} 3^{2m} 
\| \a^{\nf12} \nabla u \|_{\underline{L}^2(\cu_m)}^2 
\,.
\label{e.classical.Poincare}
\end{equation}
This follows immediately from the classical Poincar\'e in the cube~$\cu_m$ and~$\lambda^{-1} \a \leq \Id$. 
The properties of the coarse-grained matrices allow us to  generalize~\eqref{e.classical.Poincare} from uniform to coarse-grained ellipticity. In fact, we will obtain the embedding~$\mathcal{A}(\cu_m) \hookrightarrow
B^{1-s}_{2,\infty}(\cu_m)$ under the condition~$\lambda_{s,1}(\cu_m) > 0$. 

\begin{proposition}[Coarse-grained Poincar\'e inequality]
\label{p.cg.Poincare}
For every~$s,t \in (0,1]$,~$q\in [1,\infty]$,~$m \in \N$ and~$u \in \mathcal{A}(\cu_m)$,
\begin{equation}
\label{e.cg.Poincare}
\left\{
\begin{aligned}
& 3^{-sm} \left[ \nabla u  \right]_{\Besov{-s}{2}{q}(\cu_m)}
\leq 
\css{sq}^{-\nf1q} 
\lambda_{s,q}^{-\nf 12}(\cu_m) \| \a^{\nf 12} \nabla u \|_{\underline{L}^2(\cu_m)}
\,,
\\ & 
 3^{-sm} \left[ \a \nabla u  \right]_{\Besov{-s}{2}{q}(\cu_m)}
\leq
\css{sq}^{-\nf1q} 
\Lambda_{s,q}^{\nf 12}(\cu_m) \| \a^{\nf 12} \nabla u \|_{\underline{L}^2(\cu_m)}
\,.
\end{aligned}
\right.
\end{equation}
\end{proposition}
\begin{proof}
We prove only the first line of~\eqref{e.cg.Poincare}, as the second line is similar. 
By~\eqref{e.energymaps} and~\eqref{e.coarse.grained.ellipticity}, 
\begin{align*} 
3^{-sqm} [ \nabla u ]_{\Besov{-s}{2}{q}(\cu_m)}^q
& =  
\sum_{k=-\infty}^m  3^{sq (k-m)} 
\biggl( \avsum_{z \in 3^{k} \Z^d \cap \cu_m}  \bigl|(\nabla u)_{z+\cu_k} \bigr|^2
\biggr)^{\! \nf q2} 
\notag \\ &
\leq 
\sum_{k=-\infty}^m  3^{ sq (k-m)} 
\max_{z \in 3^{k} \Z^d \cap \cu_m} 
|\a_*^{-1}(z + \cu_k)|^{\nf q2} 
\biggl( 
\avsum_{z \in 3^{k} \Z^d \cap \cu_m}  \| \a^{\nf12} \nabla u \|_{\underline{L}^2(z + \cu_k)}^2
\biggr)^{\! \nf q2}
\notag \\ &
=
\css{sq}^{-1} 
\| \a^{\nf12} \nabla u \|_{\underline{L}^2(\cu_m)}^q
\css{sq} 
\sum_{k=-\infty}^m  3^{-sq (m-k)} \max_{z \in 3^{k} \Z^d \cap \cu_m} |\a_*^{-1}(z + \cu_k)|^{\nf q2} 
\notag \\ &
= 
\css{sq}^{-1} 
\lambda_{s,q}^{-\nf q2}(\cu_m) \| \a^{\nf 12} \nabla u \|_{\underline{L}^2(\cu_m)}^q
\,. & \qedhere
\end{align*}
\end{proof}

Since we assume the field~$\a(\cdot)$ is symmetric,  the first inequality in~\eqref{e.cg.Poincare} holds for~$u\in H^1_\a(\cu_m)$, and we actually have the embedding~$H^1_\a(\cu_m) \hookrightarrow
B^{1-s}_{2,\infty}(\cu_m)$ if~$\lambda_{s,1}(\cu_m) > 0$. 

\smallskip

The classical Caccioppoli inequality is a reverse Poincar\'e-type inequality: it states that, for a universal~$C<\infty$, for every coefficient field~$\a$ satisfying~\eqref{e.UE} and~$\a$-harmonic function~$u$ in~$\cu_m$, 
\begin{equation}
\| \a^{\nf12}  \nabla u \|_{\underline{L}^2(\cu_{m-1})}^2
\leq 
C \Lambda 3^{-2m}
\| u \|_{\underline{L}^2(\cu_{m})}^2
\,.
\label{e.classical.Caccioppoli}
\end{equation}
This simple inequality is the starting point for all of (divergence-form) elliptic regularity theory. 
The proof: test the equation for~$u$ with~$\zeta^2 u$ for a cutoff function~$\zeta$ to get 
\begin{equation}
\int_{\cu_m} \zeta^2 \nabla u \cdot \a\nabla u 
= - \int_{\cu_m} 2 u \zeta \nabla \zeta \cdot \a\nabla u
\,. 
\label{e.caccioppoli.testing}
\end{equation}
Applying Cauchy-Schwarz to the right side, using the trivial bound~$\nabla \zeta \cdot \a \nabla \zeta \leq \Lambda |\nabla \zeta|^2$, we get~\eqref{e.classical.Caccioppoli}. 

\smallskip

We next present the generalization of this classical inequality to coarse-grained elliptic coefficient fields. It replaces~$\Lambda$ in~\eqref{e.classical.Caccioppoli} by the coarse-grained parameter~$\Lambda_{s,1}(\cu_m)$, at the cost of an explicit prefactor that depends on~$s$,~$t$ and on the contrast~$\Lambda_{s,1}(\cu_m) / \lambda_{t,1}(\cu_m)$.  

\begin{proposition}[{Coarse-grained Caccioppoli inequality~\cite[Proposition 2.5]{AK.HC}}]
\label{p.coarse.grained.Caccioppoli}
There exists a constant~$C(d)<\infty$ such that, 
for every~$s,t\in (0,1)$ with~$s+t < 1$ and every~$u\in\A(\cu_m)$, we have
\begin{equation}
\label{e.coarse.grained.Caccioppoli}
\| \a^{\nf12}  \nabla u \|_{\underline{L}^2(\cu_{m-1})}^2
\leq 
\biggl( \frac{C}{1-s-t} \biggr)^{\! 2 + \frac{4s}{1-s-t}} 
\biggl( \frac{\Lambda_{s,1}(\cu_m)}{\lambda_{t,1}(\cu_m)} \biggr)^{\!\frac{s}{1-s-t}}
{\Lambda}_{s,1}(\cu_m) 
3^{-2m}
\| u \|_{\underline{L}^2(\cu_{m})}^2
\,.
\end{equation}
\end{proposition}

The proof of Proposition~\ref{p.coarse.grained.Caccioppoli} starts from the identity~\eqref{e.caccioppoli.testing} but estimates the right side differently. By rescaling, assume~$m=0$. Rather than using Cauchy-Schwarz and splitting the field~$\a(x)$ between the two terms, we use the duality pairing in~\eqref{e.duality.Bs} to find, with the help of~\eqref{e.divcurl.est1},
\begin{align}
\label{e.duality.Bs.cacc}
\biggl| \fint_{\cu_0} 
2 u \zeta \nabla \zeta \cdot \a\nabla u
\biggr| 
&
\leq
C
\| u \nabla (\zeta^2) \|_{\Besovul{s}{2}{\infty}(\cu_0)}
\left[ \a\nabla u \right]_{\Besov{-s}{2}{1}(\cu_{0})}
\notag\\ & 
\leq
C \| \nabla (\zeta^2) \|_{W^{1,\infty}(\cu_0)} 
\left[ \nabla u \right]_{\Besov{-(1-s)}{2}{1}(\cu_{0})}
\left[ \a\nabla u \right]_{\Besov{-s}{2}{1}(\cu_{0})}
\,.
\end{align}
The two weak norms on the right side can now be estimated using Proposition~\ref{p.cg.Poincare}, and this is where the coarse-grained ellipticity constants naturally enter. 
The actual argument is slightly more subtle than this sketch (one needs to apply the duality pairing~\eqref{e.duality.Bs} in all subcubes at a well-chosen mesoscopic scale) and we refer to~\cite{AK.HC} for the complete details. 

\smallskip

Beyond the ellipticity parameters, we also need to quantify, at each scale, the \emph{defect} of the block model relative to a fixed constant matrix~$\ahom$. The functional~$J$ conveniently packages both the Dirichlet and Neumann responses, and simultaneously the gap and the deviation from~$\ahom$. By~\eqref{e.J.a.astar} and some algebra, we find 
\begin{equation*}
J \bigl(U,\ahom^{-\nf12}e,\ahom^{\nf12}e\bigr)
=
\frac12 e \cdot 
\ahom^{-\nf12} 
\Bigl( 
\a(U) - \a_*(U) 
+
( \a_*(U) - \ahom ) 
\a_*^{-1} (U) 
( \a_*(U) - \ahom ) 
\Bigr) 
\ahom^{-\nf12}
e
\,.
\end{equation*}
This observation motivates a scale-discounted multiscale quantity that aggregates the defect across all subcubes and scales.

\begin{definition}[Multiscale defect]
\label{def.mathcalE}
For $s\in(0,\nf12)$, $q\in[1,\infty]$, and $m\in\mathbb{Z}$, define
\begin{equation}
\mathcal{E}_{s,q}(\cu_m)
\coloneqq
\biggl(
\css{sq}
\sum_{k=-\infty}^{m}
3^{-sq(m-k)}
\max_{z\in 3^k\Zd\cap \cu_m}\,
\max_{|e|=1}
\Bigl(J\bigl(z+\cu_k,\ahom^{-\nf12} e,\ahom^{\nf12} e\bigr)\Bigr)^{\nf q2} 
\biggr)^{\!\nf1q},
\label{e.mathcal.E.def}
\end{equation}
with the modification that the sum over $k$ is replaced by a supremum when $q=\infty$ and $\css{\infty}=1$.
\end{definition}

The next estimate is the coarse-grained analogue of the Poincar\'e-type bounds in Proposition~\ref{p.cg.Poincare}, now for the \emph{operator defect}. It shows that the action of $\a-\ahom$ on $\nabla u$ is small in a negative regularity norm, with size quantified by $\mathcal{E}_{s,2}$. The proof is analogous to the one of Proposition~\ref{p.cg.Poincare}. 

\begin{proposition}[{Coarse-graining an elliptic operator~\cite[Lemma 5.5]{AK.HC}}]
\label{p.CG.elliptic.operator}
There exists~$C(d)<\infty$ such that, for every~$s\in(0,\nf12)$ and every $u\in\A(\cu_m)$,
\begin{equation}
\label{e.CG.elliptic.operator}
3^{-sm} 
\bigl\| \ahom^{-\nf12} (\a-\ahom)\nabla u \bigr\|_{\Besov{-s}{2}{2} (\cu_m)}
\le
C\mathcal{E}_{s,2}(\cu_m)
\bigl\|\a^{\nf12} \nabla u\bigr\|_{\underline{L}^2(\cu_m)}
\,.
\end{equation}
\end{proposition}

Finally,~$\mathcal{E}_{s,q}$ controls the distance between solutions of the heterogeneous and constant-coefficient problems. If~$u\in H^1_\a(U)$ and~$h\in H^1(U)$ solve~$-\nabla\cdot\a\nabla u=0$ and~$-\nabla\cdot\ahom\nabla h=0$ in a Lipschitz domain~$U\subseteq\cu_0$ with the same boundary data, then
\begin{equation}
\| \ahom^{\nf12} (\nabla u-\nabla h) \|_{\Besov{-s}{2}{2} (U)}
+
\|\ahom^{-\nf12} ( \a\nabla u-\ahom\nabla h ) \|_{\Besov{-s}{2}{2} (U)}
\leq
C \mathcal{E}_{s,2}(\cu_0)
\| \a^{\nf12} \nabla u\|_{L^2(U)}
\,.
\end{equation}
The precise statement can be found in~\cite[Proposition 5.3]{AK.HC}, including the definition of the norms on the left side. 

\section{Upper bound estimate for the homogenization length scale} 
\label{s.HC}

We present a complete proof of Theorem~\ref{t.AK.HC} under the following simplifying assumptions on the law~$\P$ of the coefficient field~$\a(\cdot)$. 

\begin{enumerate}
[label=(\textrm{P\arabic*})]
\setcounter{enumi}{0}
\item \emph{Stationarity:} $\P$ is a probability measure on~$(\Omega,\F)$ which is invariant under~$\Zd$-translations. 
\label{a.stationarity}

\item\emph{Quantitative ergodicity:} $\P$ has a unit range of dependence. 
\label{a.frd} 

\item \emph{Coarse-grained ellipticity:} there exist exponents~$\lams,\Lams \in [0,1)$ and~$\xi \in [1,\infty)$ satisfying 
\begin{equation*}
\Lams+\lams <1
\qquad \mbox{and} \qquad 
16 d(1-\Lams-\lams)^{-1} \leq \xi < \infty
\end{equation*}
such that 
\label{a.ellipticity}
\begin{equation}
\E \bigl[ \Lambda_{\Lams,1}^{\xi}(\cu_0)]^{\nf1\xi}
\,
\E \bigl[ \lambda_{\lams,1}^{-\xi}(\cu_0)]^{\nf1\xi}
< \infty\,.
\label{e.ellipticity}
\end{equation}

\item \emph{Symmetry:}~$\P$--almost surely, the coefficient field~$\a(\cdot)$ is valued in~$\R^{d\times d}_{\sym}$. \label{a.symmetry}

\item \emph{Dihedral symmetry:} the law~$\P$ is invariant under permutations of the coordinate axes. That is, for every matrix~$R$ with exactly one~$1$ in each row and column and $0$s elsewhere, the law of the conjugated coefficient~$R^t \a(R \cdot) R$ is the same as that of~$\a(\cdot)$.
\label{a.iso}
\end{enumerate} 

The assumptions~\ref{a.symmetry} and~\ref{a.iso} are unnecessary and~\ref{a.frd} is much stronger than what is required (essentially any mixing condition can be used). We make these assumptions here in order to simplify the presentation, and refer the reader to~\cite[Theorem B]{AK.HC} for the more general case. 
The coarse-grained ellipticity assumption~\ref{a.ellipticity} is a variant of the condition introduced in~\cite{AK.HC}. 

\smallskip

The \emph{annealed} coarse-grained matrices are defined by
\begin{equation*}
\ahom(U) \coloneqq \E \bigl[ \a(U) \bigr] 
\qquad \mbox{and} \qquad 
\ahom_*^{-1} (U) \coloneqq \E \bigl[ \a^{-1} (U) \bigr] 
\,.
\end{equation*}
The subadditivity~\eqref{e.subadditivity} and ordering~\eqref{e.ordering} properties imply that, for every~$m,n\in\N$ with~$n\leq m$,
\begin{equation}
\label{e.monotonicity.homs}
\ahom_*(\cu_n) \leq \ahom_*(\cu_m) \leq \ahom(\cu_m) \leq \ahom(\cu_n) \,.
\end{equation}
The main use of assumption~\ref{a.iso} is that it forces the annealed matrices~$\ahom(\cu_j)$ and~$\ahom_*(\cu_j)$ to be multiplies of the identity. We will therefore treat these as scalar matrices, or positive real numbers, whichever is convenient. 

\smallskip

We introduce the exponents
\begin{equation*}
\specs_1 \coloneqq \gamma_1+ \tfrac18( 1- \Lams-\lams)
\qquad \mbox{and} \qquad  
\specs_2 \coloneqq \gamma_2 + \tfrac18( 1- \Lams-\lams)
\,.
\end{equation*}
Observe that~$\nu_1 + \nu_2 < 1$. We also introduce, for each~$n\in\N$, the quantities
\begin{equation*}
\left\{
\begin{aligned}
& \Theta_n \coloneqq  \ahom(\cu_n) \ahom_*^{-1} (\cu_n) \,, \\ 
& \tilde{\Theta}_n \coloneqq
\E \bigl[ \Lambda_{\specs_1,1}^\xi (\cu_n) \bigr]^{\nf1\xi} 
\E \bigl[ \lambda_{\specs_2,1}^{-\xi} (\cu_n) \bigr]^{\nf1\xi} 
\,.
\end{aligned}
\right.
\end{equation*}
It is clear that these quantities are monotone decreasing in~$n$ and, for every~$n$, 
\begin{equation}
\Theta_n 
\leq 
\tilde{\Theta}_n 
\leq 
\tilde{\Theta}_0
\leq 
\E \bigl[ \lambda_{\specs_1,1}^{-\xi}(\cu_0)]^{\frac1\xi}
\,
\E \bigl[ \Lambda_{\specs_2,1}^{\xi}(\cu_0)]^{\frac1\xi}
\,.
\end{equation}

The main result we prove in this section is the following theorem, which is essentially a special case of~\cite[Theorem 3.1]{AK.HC}. 

\begin{theorem}[{Homogenization after~$C\log^2 \tilde{\Theta}_0$  scales}]
\label{t.highcontrast} 
Assume~\ref{a.stationarity}--\ref{a.iso}. 
Then there exists a constant~$C(\xi,d)<\infty$ such that, for every~$\sigma \in (0,\nf12]$ and~$N\in\N$, 
\begin{equation*}
N \geq 
C \log^2 \tilde{\Theta}_0 
+
C \sigma^{-4} \log (\sigma^{-1})
\quad \implies \quad  
\tilde{\Theta}_N \leq 1+\sigma \,.
\end{equation*}
\end{theorem} 

The proof of Theorem~\ref{t.highcontrast} is based on a ``dyadic pigeonholing'' scale-selection argument: among the first~$N$ triadic scales, we extract a long run of~$h$ consecutive scales on which the relevant scale-monotone parameters vary only by fixed multiplicative factors. This selection is encoded in the following very simple dyadic (or perhaps \emph{triadic}) pigeonhole lemma. This kind of lemma is a common tool in harmonic analysis~\cite{TT}.
Note that, in the proof of Theorem~\ref{t.highcontrast},  the lemma will be iteratively applied~$C \log (1+\tilde{\Theta}_0)$ times with~$h = C \log (1+\tilde{\Theta}_0)$ and~$\delta = c$, which accounts for the~$C\log^2 (1+\tilde{\Theta}_0)$ in the statement of the theorem.

\begin{lemma}[Pigeonhole lemma] 
\label{l.pigeon}
Let~$\delta , \sigma \in (0,\nf12]$ and~$N,h \in \N$ such that~$N \geq \lceil 2 \delta^{-1} |\log \sigma|  \rceil h$. Then at least one of the following statements is valid: 
\begin{itemize} 
\item There exists a scale parameter~$n \in \{ h,\ldots, N \}$ such that
\begin{equation*}
\ahom(\cu_{n-h} ) \leq (1+\delta) \ahom(\cu_{n})
\qquad \mbox{and} \qquad 
\ahom_*^{-1} (\cu_{n-h}) \leq (1+\delta) \ahom_*^{-1} (\cu_{n} ).
\end{equation*}

\item $\Theta_N \leq \sigma \Theta_0$. 

\end{itemize} 
\end{lemma} 
\begin{proof} 
Let $k\coloneqq  \lceil 2 \delta^{-1} |\log \sigma|  \rceil$. 
By the monotonicity of~$n\mapsto \ahom(\cu_n)$ and~$n\mapsto \ahom_*^{-1}(\cu_n)$ in~\eqref{e.monotonicity.homs} and~$N\geq kh$, if the first alternative fails, then
\begin{align*}
\Theta_N \Theta_0^{-1} 
&
=
\ahom(\cu_N)  \ahom^{-1} (\cu_0) 
\ahom_*^{-1} (\cu_N)
\ahom_* (\cu_0)
\notag \\ & 
\leq 
\prod_{j=1}^k 
\ahom(\cu_{jh})  \ahom^{-1} (\cu_{(j-1)h}) 
\ahom_*^{-1} (\cu_{jh})
\ahom_* (\cu_{(j-1)h} ) 
\leq 
(1+\delta)^{-k}
\leq \sigma 
\,.
\end{align*}
This completes the proof.  
\end{proof} 

We introduce the  \emph{expected additivity defect} between scales~$k$ and~$m$ by
\begin{align*}
\tau_{n,k}(p,q) 
\coloneqq 
& \
\E \bigl[  J(\cu_k, p,q) \bigr]
-
\E \bigl[ J(\cu_m,p,q) \bigr]
\notag \\ 
= & \
\frac12 p\cdot \bigl( \ahom(\cu_k) - \ahom(\cu_m) \bigr) p 
+
\frac12 q\cdot \bigl( \ahom_*^{-1}(\cu_k) - \ahom_*^{-1}(\cu_m)  \bigr) q
\,. 
\end{align*}
In the next lemma, we give an upper bound for the expectation of~$J(U,p,q)$ in terms of~$\tau_{n,k}(p,q)$ and of weak norms of the gradient and flux of its maximizer. 

\begin{lemma}[Upper bound of $J$ by weak norms]
\label{l.J.upperbound} 
There exists a constant~$C(d)<\infty$ such that, for every~$m,k\in\Z$ with~$k<m$,~$p,q,p_0,q_0\in\Rd$ and~$s,t\in (0,1)$ with~$s+t\leq 1$,  
\begin{align*}
\lefteqn{ 
\E \bigl[ J(\cu_m, p , q )\bigr] - \frac12 p_0 \cdot q_0 
} \ \ & 
\notag \\ &
\leq
C\tau_{n,k}^{\nf12}(p,q)  \E \bigl[ 
J(\cu_k, p,q) 
\bigr]^{\nf12} 
+
C 3^{-(m-k)} 
\E \bigl[ J(\cu_m,p,q) \bigr]
\notag \\ & \quad 
+
C|q_0|  \E \bigl[ 3^{-sm}  \| \nabla v(\cdot,\cu_m,p,q) - p_0 \|_{\Besov{-s}{2}{1}(\cu_m)}\bigr]
+
C|p_0|  \E \bigl[ 3^{-tm} \| \a \nabla v(\cdot,\cu_m,p,q) - q_0 \|_{\Besov{-t}{2}{1}(\cu_m)} \bigr]
\notag \\ & \quad 
+
C
\E \bigl[ 
3^{-2sm} 
\| \nabla v(\cdot,\cu_m,p,q) - p_0  \|_{\Besov{-s}{2}{1}(\cu_m)}^2 
\bigr]^{\nf12} 
\E \bigl[
3^{-2tm} 
\| \a \nabla v(\cdot,\cu_m,p,q) - q_0 \|_{\Besov{-t}{2}{1}(\cu_m)}^2 
\bigr]^{\nf12}  
\,.
\end{align*}
\end{lemma}
\begin{proof}
Denote~$v_m \coloneqq v(\cdot,\cu_m,p,q)$ and~$v_{k,z}\coloneqq v(\cdot,z+\cu+k,p,q)$. 
Select a nonnegative, smooth function~$\varphi\in C^\infty_c(\cu_m)$ satisfying 
\begin{equation*}
(\varphi)_{\cu_m} = 1
\qquad \mbox{and} \qquad 
\| \varphi \|_{L^\infty(\cu_m)} 
+ 
3^m 
\| \nabla \varphi \|_{L^\infty(\cu_m)} 
\leq C \,.
\end{equation*}
By~\eqref{e.J.energy.of.v},~\eqref{e.J.second.var} and~$(\varphi)_{\cu_m} = 1$, we have the identity
\begin{align}
\lefteqn{
J(\cu_m,p,q) - \frac12p_0 \cdot q_0 
=
\fint_{\cu_m} 
\frac12 \nabla v_m \cdot \a\nabla v_m - \frac12p_0 \cdot q_0
} \qquad &
\notag \\ & 
=
\fint_{\cu_m} 
\frac12 \varphi \bigl( \nabla v_m - p_0) 
\cdot ( \a\nabla v_m - q_0) 
+
\avsum_{z\in 3^k\Zd\cap \cu_m}
\fint_{z+\cu_k} 
\frac12 
\bigl( ( \varphi )_{z+\cu_k} - \varphi  \bigr)
\nabla v_m \cdot \a\nabla v_m
\notag \\ & \qquad 
+
\avsum_{z\in 3^k\Zd\cap \cu_m} \! \! \! 
\bigl( 1 - ( \varphi )_{z+\cu_k} \bigr) 
 \fint_{z+\cu_k} \frac12 \bigl(
( \nabla v_m - \nabla v_{k,z} ) \cdot \a ( \nabla v_m + \nabla v_{k,z} ) +   \nabla v_{k,z} \cdot \a  \nabla v_{k,z}  \bigr)
\notag \\ & \qquad 
+ 
\frac12 q_0 \cdot \fint_{\cu_m} \varphi( \nabla v_m - p_0) 
+
\frac12 p_0 \cdot \fint_{\cu_m} \varphi( \a \nabla v_m - q_0) 
\,.
\label{e.centered.J.splitting}
\end{align}
We estimate the first term on the right side of~\eqref{e.centered.J.splitting} by using the equation for~$v_m$ and integrating by parts (this is the point of introducing the cutoff function~$\varphi$). Applying~\eqref{e.duality.Bs},~\eqref{e.divcurl.est1} and the assumption that~$s+t\leq1$, we get
\begin{align*}
\fint_{\cu_m} 
\frac12 \varphi \bigl( \nabla v_m - p_0) 
\cdot ( \a\nabla v_m - q_0) 
&
=
-
\fint_{\cu_m} 
\frac12
( v_m - \ell_{p_0}) 
\nabla \varphi
\cdot ( \a\nabla v_m - q_0) 
\notag \\ & 
\leq 
C
\| \nabla \varphi( v_m - \ell_{p_0}) \|_{\Besovul{1-s}{2}{\infty}(\cu_m)}
\| \a \nabla v_m - q_0 \|_{\Besov{s-1}{2}{1}(\cu_m)}
\notag \\ & 
\leq 
C3^{-m}
\| \nabla v_m - p_0   \|_{\Besov{-s}{2}{1}(\cu_m)}
\| \a \nabla v_m - q_0 \|_{\Besov{s-1}{2}{1}(\cu_m)}
\notag \\ & 
\leq 
C3^{-sm}
\| \nabla v_m - p_0  \|_{\Besov{-s}{2}{1}(\cu_m)}
3^{-tm} 
\| \a \nabla v_m - q_0 \|_{\Besov{-t}{2}{1}(\cu_m)}
\,.
\end{align*}
We estimate the second term on the right side of~\eqref{e.centered.J.splitting} by
\begin{align*}
\lefteqn{ 
\biggl| 
\avsum_{z\in 3^k\Zd\cap \cu_m}
\fint_{z+\cu_k} 
\frac12 
\bigl( ( \varphi )_{z+\cu_k} - \varphi  \bigr)
\nabla v_m \cdot \a\nabla v_m
\biggr| 
} \qquad & 
\notag \\ &  
\leq 
\max_{z\in 3^k\Zd\cap \cu_m}
\bigl\| \varphi - ( \varphi )_{z+\cu_k}  \bigr\|_{L^\infty(z+\cu_k)} 
\fint_{\cu_m} 
\frac12 
\nabla v_m \cdot \a\nabla v_m
\leq
C 3^{-(m-k)} 
J(\cu_m,p,q) 
\,.
\end{align*}
Using Cauchy-Schwarz,~\eqref{e.J.energy.of.v} and~\eqref{e.subadditivity}, we estimate the fourth term on the right side of~\eqref{e.centered.J.splitting} by
\begin{align*}
\lefteqn{ 
\biggl| 
\avsum_{z\in 3^k\Zd\cap \cu_m} 
\bigl( 1 - ( \varphi )_{z+\cu_k} \bigr) 
\fint_{z+\cu_k} \frac12
( \nabla v_m - \nabla v_{k,z} ) \cdot \a ( \nabla v_m + \nabla v_{k,z} )
\biggr| 
} 
\qquad & 
\notag \\ & 
\leq 
\| \varphi \|_{L^\infty(\cu_m)} 
\biggl( 
\avsum_{z\in 3^k\Zd\cap \cu_m} 
\| \a^{\nf12} ( \nabla v_m - \nabla v_{k,z} ) \|_{\underline{L}^2(z+\cu_k)}^2 
\biggr)^{\!\nf12}
\notag \\ & \qquad \times  
\biggl( 
\| \a^{\nf12} \nabla v_m\|_{\underline{L}^2(\cu_m)}^2  
+
\avsum_{z\in 3^k\Zd\cap \cu_m} 
\| \a^{\nf12}\nabla v_{k,z}  \|_{\underline{L}^2(z+\cu_k)}^2
\biggr)^{\!\nf12} 
\notag \\ & 
\leq 
C
\biggl( 
\avsum_{z\in 3^k \Zd\cap \cu_m}
J(z+\cu_k, p,q) 
-
J(\cu_m,p,q)  
\biggr)^{\!\nf12} 
\biggl( \avsum_{z\in 3^k \Zd\cap \cu_m}
J(z+\cu_k, p,q) 
\biggr)^{\!\nf12}\,.
\end{align*}
Using~\eqref{e.duality.Bs}, we estimate the two terms on the last line of~\eqref{e.centered.J.splitting}, by 
\begin{align*}
\biggl| \frac12 q_0 \cdot \fint_{\cu_m} \varphi( \nabla v_m - p_0) 
\biggr| 
& 
\leq 
C |q_0| 
\| \varphi \|_{\Besovul{s}{2}{\infty}(\cu_m)}
\| \nabla v_m - p_0 \|_{\Besov{-s}{2}{1}(\cu_m)} 
\leq 
C|q_0| 3^{-sm} \| \nabla v_m - p_0 \|_{\Besov{-s}{2}{1}(\cu_m)} 
\end{align*}
and, similarly, 
\begin{equation*}
\biggl| 
\frac12 p_0 \cdot \fint_{\cu_m} \varphi( \a \nabla v_m - q_0) 
\biggr| 
\leq 
C|p_0| 3^{-tm} \| \a \nabla v_m - q_0 \|_{\Besov{-t}{2}{1}(\cu_m)}
\,.
\end{equation*}
Combining the above with~\eqref{e.centered.J.splitting} yields, using again~\eqref{e.J.energy.of.v}, 
\begin{align}
\lefteqn{ 
\biggl|
J(\cu_m,p,q) - \frac12p_0 \cdot q_0 
-
\avsum_{z\in 3^k\Zd\cap \cu_m} 
\bigl( 1 - ( \varphi )_{z+\cu_k} \bigr) 
J(z+\cu_k, p,q) 
\biggr|
} 
\qquad &   
\notag \\ & 
\leq
C
\biggl( 
\avsum_{z\in 3^k \Zd\cap \cu_m}
J(z+\cu_k, p,q) 
-
J(\cu_m,p,q)  
\biggr)^{\!\nf12} 
\!
\biggl( \avsum_{z\in 3^k \Zd\cap \cu_m}
J(z+\cu_k, p,q) 
\biggr)^{\!\nf12}
\notag \\ & \qquad 
+
C 3^{-(m-k)} 
J(\cu_m,p,q) 
+
C3^{-sm}
\| \nabla v_m - p_0  \|_{\Besov{-s}{2}{1}(\cu_m)}
3^{-tm} 
\| \a \nabla v_m - q_0 \|_{\Besov{-t}{2}{1}(\cu_m)}
\notag \\ & \qquad
+
C|q_0| 3^{-sm} \| \nabla v_m - p_0 \|_{\Besov{-s}{2}{1}(\cu_m)} 
+
C|p_0| 3^{-tm} \| \a \nabla v_m - q_0 \|_{\Besov{-t}{2}{1}(\cu_m)}
\,.
\label{e.Jtilde.energy.bound.no.E}
\end{align}
By taking the expectation, using~$(\varphi)_{\cu_m} = 1$,~\eqref{e.J.energy.of.v} and stationarity assumption~\ref{a.stationarity}, we obtain 
\begin{equation*}
\E \biggl[
\avsum_{z\in 3^k\Zd\cap \cu_m} 
\bigl( 1 - ( \varphi )_{z+\cu_k} \bigr) 
J(z+\cu_k, p,q)  
\biggr] = 
\E \bigl[ J(\cu_k,p,q) \bigr] 
\avsum_{z\in 3^k\Zd\cap \cu_m} \! \! \! 
\bigl( 1 - ( \varphi )_{z+\cu_k} \bigr) 
= 0\,,
\end{equation*}
and thus also the lemma. 
\end{proof}

Motivated by the previous lemma, we estimate the weak norms of the gradient and flux of the maximizer of~$J(\cu_m,p,q)$. 

\begin{lemma}[Weak norm estimates]
\label{l.weaknorms}
There exists~$C(d)<\infty$ such that, for every~$s,s' \in (0,1]$ with~$s' \in [\frac12 s,s]$,~$m,k\in\N$ with~$k<m$ and~$p,q,p_0,q_0\in\Rd$, we have the estimates
\begin{align}
\lefteqn{
3^{-sm} 
\| \nabla v(\cdot,\cu_m,p,q) - p_0  \|_{\Besov{-s}{2}{1}(\cu_m)} 
} \qquad & 
\notag \\ &
\leq 
\sum_{j=k+1}^m 
3^{-s(m-j)}
\biggl(
\avsum_{z\in 3^j\Zd \cap \cu_m}
\bigl| (\a_*^{-1}(z+\cu_j) q - p ) - p_0 \bigr |^2
\biggr)^{\!\nicefrac 12}
\notag \\ & \qquad 
+
C
\lambda_{s',1}^{-\nf12} (\cu_m) 
\sum_{j=k+1}^m 
3^{-(s-s') (m-j)}
\biggl( \avsum_{z\in 3^j\Zd \cap \cu_m}
J(z+\cu_j,p,q) 
-
J(\cu_m,p,q)
\biggr)^{\!\nf12}
\notag \\ & \qquad 
+
C s^{-\nf 12} 3^{-(s-s')(m-k)} 
\lambda_{s',1}^{-\nf12} (\cu_m) 
J(\cu_m,p,q)^{\nf12} 
+
C s^{-1} 3^{-s(m-k)} |p_0| 
\label{e.weak.norms.gradient.bound}
\end{align}
and
\begin{align}
\lefteqn{
3^{-sm} 
\| \a\nabla v(\cdot,\cu_m,p,q) - q_0  \|_{\Besov{-s}{2}{1}(\cu_m)} 
} \qquad & 
\notag \\ &
\leq 
\sum_{j=k+1}^m 
3^{-s(m-j)}
\biggl(
\avsum_{z\in 3^j\Zd \cap \cu_m}
\bigl|  ( q - \a(z+\cu_j) p ) - q_0 \bigr |^2
\biggr)^{\!\nicefrac 12}
\notag \\ & \qquad 
+
C
\Lambda_{s',1}^{\nf12} (\cu_m) 
\sum_{j=k+1}^m 
3^{-(s-s') (m-j)}
\biggl( 
\avsum_{z\in 3^j\Zd \cap \cu_m}
J(z+\cu_j,p,q) 
-
J(\cu_m,p,q)
\biggr)^{\!\nf12}
\notag \\ & \qquad 
+
Cs^{-\nf 12} 3^{-(s-s')(m-k)}  
\Lambda_{s',1}^{\nf12} (\cu_m) 
J(\cu_m,p,q)^{\nf12} 
+
C s^{-1} 3^{-s(m-k)} |q_0| 
\,.
\label{e.weak.norms.flux.bound}
\end{align}
\end{lemma}
\begin{proof}
Denote~$v_m \coloneqq v(\cdot,\cu_m,p,q)$ and~$v_{j,z}\coloneqq v(\cdot,j+\cu+k,p,q)$ for~$j<m$. 
We compute
\begin{align*}
3^{-sm} 
\| \nabla v_m - p_0  \|_{\Besov{-s}{2}{1}(\cu_m)} 
&
=
\sum_{j=-\infty}^m
3^{-s (m-j)}
\biggl(
\avsum_{z\in 3^j\Zd \cap \cu_m}
\bigl| (\nabla v_m )_{z+\cu_j} - p_0 \bigr |^2
\biggr)^{\!\nicefrac 12}
\notag \\ & 
\leq
\sum_{j=k+1}^m 
3^{-s(m-j)} 
\biggl(
\avsum_{z\in 3^j\Zd \cap \cu_m}
\bigl| (\nabla v_m )_{z+\cu_j} - p_0 \bigr |^2
\biggr)^{\!\nicefrac 12}
\notag \\ & \quad
+
\sum_{j=-\infty}^{k}
3^{-s(m-j)} 
\biggl(
\avsum_{z\in 3^j\Zd \cap \cu_m}
\bigl| (\nabla v_m )_{z+\cu_j} \bigr|^2 
\biggr)^{\!\nicefrac 12}
+ Cs^{-1} 3^{-s(m-k)} |p_0| 
\,.
\end{align*}
To bound the first term on the right side, we use that, for each~$j \in \Z\cap(-\infty, m)$, 
\begin{align*}
\biggl(
\avsum_{z\in 3^j\Zd \cap \cu_m}
\bigl| (\nabla v_m)_{z+\cu_j} - p_0 \bigr |^2
\biggr)^{\!\nicefrac 12}
&
\leq 
\biggl(
\avsum_{z\in 3^j\Zd \cap \cu_m}
\bigl| (\a_*^{-1}(z+\cu_j) q - p ) - p_0 \bigr |^2
\biggr)^{\!\nicefrac 12}
\notag \\ & \qquad 
+
\biggl(
\avsum_{z\in 3^j\Zd \cap \cu_m}
\bigl| (\nabla v_m )_{z+\cu_j} - (\nabla v_{j,z} )_{z+\cu_j} \bigr |^2
\biggr)^{\!\nicefrac 12}
\end{align*}
and, by~\eqref{e.energymaps},
\begin{align*}
\avsum_{z\in 3^j\Zd \cap \cu_m}
\bigl| (\nabla v_m - \nabla v_{j,z} )_{z+\cu_j} \bigr |^2
&
\leq 
\avsum_{z\in 3^j\Zd \cap \cu_m}
| \a_*^{-1} (z{+}\cu_j) |
\| \a^{\nf12} (\nabla v_m - \nabla v_{j,z} ) \|_{\underline{L}^2(z+\cu_j)}^2
\notag \\ & 
\leq 
\max_{z\in 3^j\Zd \cap \cu_m}
\!
| \a_*^{-1} (z{+}\cu_j) |
\biggl( \avsum_{z\in 3^j\Zd \cap \cu_m}
J(z{+}\cu_j,p,q) 
-
J(\cu_m,p,q)
\biggr)
\,.
\end{align*}
To bound the second term on the right side, we use 
\begin{align*}
\avsum_{z\in 3^j\Zd \cap \cu_m}
\bigl| (\nabla v_m )_{z+\cu_j} \bigr |^2
&
\leq 
\avsum_{z\in 3^j\Zd \cap \cu_m}
| \a_*^{-1} (z+\cu_j) |
\| \a^{\nf12}\nabla v_m \|_{\underline{L}^2(z+\cu_j)}^2
\notag \\ & 
\leq 
\max_{z\in 3^j\Zd \cap \cu_m}
| \a_*^{-1} (z+\cu_j) |
\| \a^{\nf12}\nabla v_m \|_{\underline{L}^2(\cu_m)}^2
\,.
\end{align*}
Combining these with the definition~\eqref{e.coarse.grained.ellipticity} and~$\lambda_{s',\infty}^{-1} \leq  \lambda_{s',1}^{-1}$ yields~\eqref{e.weak.norms.gradient.bound}. The proof of~\eqref{e.weak.norms.flux.bound} is entirely analogous, using~\eqref{e.energymaps.flux} in place of~\eqref{e.energymaps}, so we omit it. 
\end{proof}

We next show that the quantity~$\tilde{\Theta}_m$ can be controlled by~$\Theta_0$ and a scale-separation error. 

\begin{lemma} 
\label{l.lambda.bounds}
There exists a constant~$C(d)<\infty$ such that, for every~$s \in (0,1)$ and~$m \in \N$, we have
\begin{equation} 
\label{e.lambda.moment.bound}
\E \Bigl[ \bigl(\lambda_{s,1}^{-1} (\cu_m) -\ahom_*^{-1}(\cu_0)   \bigr)_+^{\xi} \Bigr]^{\nf 1\xi}
\leq 
\frac{ C \xi }{d-2s}   3^{-(s-\frac d\xi) m} 
\E\bigl[ \lambda_{s,1}^{-\xi} (\cu_0) \bigr]^{\nf 1\xi} 
\end{equation}
and
\begin{equation} 
\label{e.Lambda.moment.bound}
\E \Bigl[ \bigl( \Lambda_{s,1} (\cu_m) - \ahom(\cu_0)  \bigr)_+^\xi \Bigr]^{\nf 1\xi}
\leq
\frac{ C \xi }{d-2s}   3^{-(s-\frac d\xi) m} 
\E\bigl[ \Lambda_{s,1}^{\xi} (\cu_0) \bigr]^{\nf 1\xi} 
\,.
\end{equation}
In particular, 
\begin{equation} 
\label{e.tilde.Theta.vs.Theta}
\tilde{\Theta}_m
\leq
\Theta_0 
+
C\xi^2 3^{-(\specs_1 \wedge \specs_2-\frac d\xi) m} 
\tilde{\Theta}_0
\,.
\end{equation}
\end{lemma}
\begin{proof}
We start by estimating 
\begin{equation*} 
\lambda_{s,1}^{-1} (\cu_m)
\leq
\cs \sum_{n=1}^m 3^{-s(m-n)} \max_{z \in 3^n \Zd \cap \cu_m}\bigl| \a_*^{-1}(z+\cu_n) \bigr|
+
3^{-sm} \max_{z \in \Zd \cap \cu_m} \lambda_{s,1}^{-1} (z+\cu_0)
\,.
\end{equation*}
By subadditivity, we get, for every~$n\in \N$, 
\begin{equation*} 
\bigl| \a_*^{-1}(z+\cu_n) \bigr|
\leq
\biggl| \avsum_{z' \in z+ \Zd \cap \cu_n} \a_*^{-1}(z'+\cu_0)  \biggr|
\leq
\ahom_*^{-1}(\cu_0)
+
\biggl|  \avsum_{z' \in z+ \Zd \cap \cu_n} \a_*^{-1}(z'+\cu_0)  - \ahom_*^{-1}(\cu_0) \biggr|
\end{equation*}
and then, by Rosenthal's inequality (a concentration inequality for finite moments of sums of independent random variables, see~\cite[Theorem 15.11]{BLM}), 
\begin{align*} 
\lefteqn{
\E\biggl[ \Bigl| \avsum_{z \in \Zd \cap \cu_n} \a_*^{-1}(z+\cu_0)  - \ahom_*^{-1}(\cu_0) \Bigr|^{\xi} \biggr]^{\nf 1\xi}
} \qquad & 
\notag \\ &
\leq C \xi^{\nf 12} 3^{-\frac d2 n } \E\bigl[ |\a_*^{-1}(\cu_0)|^2  \bigr]^{\nf 12}
+ C \xi 3^{-d(1-\nf 1\xi) n } \E\bigl[ |\a_*^{-1}(\cu_0)|^\xi  \bigr]^{\nf 1\xi}
\leq C \xi 3^{-\frac d2 n } \E\bigl[ |\a_*^{-1}(\cu_0)|^\xi  \bigr]^{\nf 1\xi}
\,.
\end{align*}
Thus, by a union bound 
\begin{align*} 
\notag
\lefteqn{
\E\biggl[\biggl( \cs \sum_{n=1}^m 3^{-s(m-n)} \max_{z \in 3^n \Zd \cap \cu_m} \a_*^{-1}(z+\cu_n) - \ahom_*^{-1}(\cu_0)\biggr)_{\!+}^{\!\xi}\biggr]^{\nf 1\xi}
} \qquad &
\notag \\ &
\leq
C \xi s 3^{-(s-\frac d\xi) m}  \E\bigl[ |\a_*^{-1}(\cu_0)|^\xi  \bigr]^{\nf 1\xi} \sum_{n=1}^m 3^{-(\frac d2  + \frac d\xi - s)n} 
\leq
\frac{C \xi s}{d-2s}  3^{-(s-\frac d\xi) m} \E\bigl[ \lambda_{s,1}^{-\xi} (\cu_0) \bigr]^{\nf 1\xi} 
\,.
\end{align*}
By another union bound,
\begin{equation*} 
3^{-sm} \E\biggl[ \max_{z \in \Zd \cap \cu_m} \lambda_{s,1}^{-\xi} (z+\cu_0) \biggr]^{\nf 1\xi} 
\leq
C 3^{-(s-\frac d\xi) m} \E\bigl[ \lambda_{s,1}^{-\xi} (\cu_0) \bigr]^{\nf 1\xi} 
\,.
\end{equation*}
Combining the above displays yields~\eqref{e.lambda.moment.bound}.
The proof of~\eqref{e.Lambda.moment.bound} is analogous and we omit it.

\smallskip 

Finally, by~$\ahom_*^{-1}(\cu_0) \leq \E[\lambda_{\nu_1,1}^{-1}(\cu_0)]$ and~$\ahom(\cu_0) \leq \E[\Lambda_{\nu_2,1}(\cu_0)]$, 
\begin{align*} 
\tilde{\Theta}_m & =
\E \bigl[ \lambda_{\specs_1,1}^{-\xi} (\cu_m) \bigr]^{\nf 1\xi}
\E \bigl[ \Lambda_{\specs_21}^{\xi} (\cu_m) \bigr]^{\nf 1\xi}
\notag \\ &
\leq
\Bigl(\E \Bigl[ \bigl(\lambda_{\specs_1,1}^{-1} (\cu_m) -\ahom_*^{-1}(\cu_0)   \bigr)_+^{\xi} \Bigr]^{\nf 1\xi} + \ahom_*^{-1}(\cu_0) \Bigr)
\Bigl(
\E \Bigl[ \bigl( \Lambda_{\specs_2,1} (\cu_m) - \ahom(\cu_0)  \bigr)_+^\xi \Bigr]^{\nf 1\xi}
+\ahom(\cu_0) \Bigr)
\notag \\ &
\leq
\Theta_0
+
C\xi^2 3^{-(\specs_1 \wedge \specs_2-\frac d\xi) m} 
\tilde{\Theta}_0
 \,,
\end{align*}
and thus~\eqref{e.tilde.Theta.vs.Theta} follows. The proof is complete.
\end{proof}

We combine the previous lemmas to obtain the contraction of~$\Theta_m$ for~$m\geq C\log\tilde{\Theta}_0$, in the case that the first alternative of Lemma~\ref{l.pigeon} is valid.

\begin{lemma}[{Contraction of $\Theta_m$ after~$C\log\tilde{\Theta}_0$ scales}]
\label{l.toss.J}
There exists a constant~$C(d)<\infty$ such that, for every~$\delta \in (0,\nf 12]$ and~$m\in\N$ satisfying 
\begin{equation}
m \geq \xi \log \bigl(
\delta^{-1} \xi  \tilde{\Theta}_0  
\bigr) 
\label{e.big.jump}
\end{equation} 
and
\begin{equation}
\ahom(\cu_0) \leq (1+\delta) \ahom(\cu_m) 
\quad \mbox{and} \quad 
\ahom_*^{-1} (\cu_0) \leq (1+\delta) \ahom_*^{-1} (\cu_m) \,,
\label{e.hypothesis.squish}
\end{equation}
we have the estimate
\begin{equation}
\Theta_m 
\leq 
1 
+ 
C\xi \delta^{\nicefrac14} \Theta_0
\,.
\label{e.Theta.m.bound}
\end{equation}
\end{lemma}
\begin{proof}
We fix~$\delta \in (0,\nf 12]$ and~$m \in \N$, and assume that both~\eqref{e.big.jump} and~\eqref{e.hypothesis.squish} are valid. Let
\begin{equation*} 
s \coloneqq \specs_1 + \frac18 (1-\gamma_1 - \gamma_2) 
\qand
t \coloneqq
\specs_2 + \frac18 (1-\gamma_1 - \gamma_2) 
  \,,
\end{equation*}
so that~$t+s=\frac12 (1+\gamma_1 + \gamma_2) < 1$. We introduce the auxiliary quantity
\begin{align}
\widetilde{J}(U,p,q)
&
\coloneqq  
J(U,p,q) - 
\frac12 
( \ahom_*^{-1}(U)q - p )
\cdot 
( q -  \ahom(U) p )
\notag \\ & \;
= 
\frac12 p \cdot \bigl( \a(U) - \ahom(U) \bigr) p 
+ \frac12 q \cdot \bigl( \a_*^{-1} (U) - \ahom_*^{-1} (U) \bigr) q 
+ \frac12 p\cdot ( \ahom(U) \ahom_*^{-1}(U) - \Id ) q 
\,,
\label{e.J.tilde.def}
\end{align}
where in the last line we used~\eqref{e.J.a.astar}. Taking expectations, the first two terms on the right vanish and we obtain 
\begin{equation*}
\E \bigl[ \widetilde{J}(U,p,q) \bigr] 
=
\frac12 p\cdot ( \ahom(U) \ahom_*^{-1}(U) - \Id ) q 
\,.
\end{equation*}
We deduce that, for any~$\mu\in (0,\infty)$, 
\begin{equation}
\bigl| \ahom(U) \ahom_*^{-1}(U) - \Id \bigr| 
\leq 
2 \max_{|e|=1} 
\E \bigl[ \widetilde{J}(U, \mu^{-1} e , \mu e ) \bigr] 
\,.
\label{e.ahom.astarhom.ratio}
\end{equation}
We apply this inequality with~$\mu \coloneqq \m_0^{\nf12}$, where~$\m_0$ is the geometric mean of~$\ahom(\cu_m)$ and~$\ahom_*(\cu_m)$,
\begin{equation}
\m_0 \coloneqq 
\bigl( \ahom(\cu_m) \ahom_*(\cu_m) \bigr)^{\nf12} 
\,.
\label{e.m0.choice}
\end{equation}
Thus the left side of~\eqref{e.hypothesis.squish} is bounded by 
\begin{equation}
\Theta_m = 
\ahom(\cu_m) \ahom_*^{-1}(\cu_m)
\leq
1 
+
2 \max_{|e|=1} 
\E \bigl[ \widetilde{J}\bigl(\cu_m, \m_0^{-\nf12} e , \m_0^{\nf12} e \bigr) \bigr] 
\,.
\label{e.setup}
\end{equation}
To bound the right side of~\eqref{e.ahom.astarhom.ratio}, we fix~$|e|=1$ and apply Lemmas~\ref{l.J.upperbound} and~\ref{l.weaknorms}, with the choices:
\begin{equation}
\left\{
\begin{aligned}
& p \coloneqq \m_0^{-\nf12} e 
\,, \\  
& q \coloneqq \m_0^{\nf12} e  = \m_0 p
\,, \\  
& p_0 \coloneqq \ahom_*^{-1}(\cu_m) q - p 
= \m_0^{-\nf12} \bigl( \ahom^{\nf12}(\cu_m) \ahom_*^{-\nf12} (\cu_m) - \Id \bigr)e
\,, \\ 
& 
q_0 \coloneqq q - \ahom(\cu_m) p 
= 
\m_0^{\nf12} \bigl( \Id -   \ahom^{\nf12}(\cu_m) \ahom_*^{-\nf12} (\cu_m) \bigr)e
= -\m_0 p_0
\,.
\end{aligned}
\right.
\label{e.parameter.choices}
\end{equation}
Before we apply the lemmas, we record some preliminary estimates. First, observe that the hypothesis~\eqref{e.hypothesis.squish} and our choices in~\eqref{e.m0.choice} and~\eqref{e.parameter.choices} yield, after some straightforward computations, 
\begin{equation}
\bigl| \m_0^{-\nf12} q_0 \bigr| 
=
\bigl| \m_0^{\nf12} p_0 \bigr| 
=
\bigl| 
\ahom^{\nf12}(\cu_m) \ahom_*^{-\nf12} (\cu_m) - \Id 
\bigr| 
\leq 
\Theta_0^{\nf12} 
\,,
\label{e.bounds.for.p0.q0}
\end{equation}
\begin{equation}
\m_0^{-1} \ahom(\cu_0) 
\leq (1+\delta) \Theta_0^{\nf12} 
\qquad \mbox{and} \qquad 
\m_0 \ahom_{*}^{-1} (\cu_0) 
\leq (1+\delta) \Theta_0^{\nf12} \,,
\label{e.compare.m0.ahom0}
\end{equation}
and, for~$k\in \N$ with~$k \leq m$, 
\begin{equation}
\E \bigl[ J(\cu_k,p,q) \bigr]
\leq 
\E \bigl[  J(\cu_0, p,q) \bigr]
=
\frac12 e \cdot 
\bigl( \m_0^{-1} \ahom(\cu_0) 
+ 
\m_0 \ahom_*^{-1} (\cu_0) 
- 2
\bigr)
e 
\leq 
(1+\delta) 
\Theta_0^{\nf12}
\label{e.J.bounds.squish}
\end{equation}
and
\begin{align}
\tau_{m,k}(p,q) 
&
=
\E \bigl[  J(\cu_k, p,q) \bigr]
-
\E \bigl[ J(\cu_m,p,q) \bigr]
\notag \\ &
=
\frac12 e \cdot 
\bigl( \m_0^{-1} ( \ahom(\cu_k) -\ahom(\cu_m) ) 
+ 
\m_0 ( \ahom_*^{-1} (\cu_k) - \ahom_*^{-1}(\cu_m)) 
\bigr) e
\leq 
\delta \Theta_0^{\nf12} 
\,.
\label{e.tau.mk.pq}
\end{align}
We also have the moment bound
\begin{equation} 
\label{e.J.moments}
\E\bigl[ J(\cu_0,p,q)^\xi\bigr]^{\nf 1\xi}
\leq 
\frac12 \m_0^{-1} \E\bigl[ | \a(\cu_0)|^\xi\bigr]^{\nf 1\xi}  
+
\frac12 \m_0 \E\bigl[ | \a_*^{-1}(\cu_0)|^\xi\bigr]^{\nf 1\xi}  
\leq 
\tilde{\Theta}_0
\,.
\end{equation}
In view of~\eqref{e.lambda.moment.bound} and~\eqref{e.Lambda.moment.bound},~\eqref{e.big.jump} implies that 
\begin{equation} 
\label{e.lambda.bounds.applied}
\ahom_*(\cu_0) \E\bigl[ \lambda_{\specs_1,1}^{-\xi}(\cu_m) \bigr]^{\nf 1\xi} + \ahom^{-1}(\cu_0) \E\bigl[ \Lambda_{\specs_2,1}^\xi(\cu_m) \bigr]^{\nf 1\xi}
\leq 
C
\,.
\end{equation} 
We are ready to apply Lemma~\ref{l.J.upperbound}. Fix~$\ep > 0$ to be selected below.  Lemma~\ref{l.J.upperbound}, Cauchy's inequality,~\eqref{e.bounds.for.p0.q0},~\eqref{e.compare.m0.ahom0},~\eqref{e.J.bounds.squish} and~\eqref{e.tau.mk.pq} yield
\begin{align}
\E\bigl  [ \tilde{J} (\cu_m, p , q ) \bigr]
&
=
\E\Bigl  [ {J} (\cu_m, p , q ) - \frac12 p_0 \cdot q_0 \Bigr]
\notag \\ & 
\leq
\bigl(\ep + C 3^{-m} + C \ep^{-1} \delta^2 \bigr)\Theta_0 
\notag \\ & \qquad 
+
C\ep^{-1} \m_0
3^{-2sm} \E \Bigl[  \| \nabla v(\cdot,\cu_m,p,q) - p_0 \|_{\Besov{-s}{2}{1}(\cu_m)}^2 \Bigr] 
\notag \\ & \qquad 
+
C\ep^{-1} \m_0^{-1} 3^{-2tm} \E \Bigl[  \| \a \nabla v(\cdot,\cu_m,p,q) - q_0 \|_{\Besov{-t}{2}{1}(\cu_m)}^2  \Bigr] 
\,.
\label{e.J.tilde.initial.plug}
\end{align}
To estimate the weak norms on the last two lines of~\eqref{e.J.tilde.initial.plug}, we apply Lemma~\ref{l.weaknorms}. We claim that
\begin{multline} 
\label{e.weak.norms.in.the.proof}
\ahom_*(\cu_0) 
\E \bigl[ 
3^{-2sm} 
\| \nabla v(\cdot,\cu_m,p,q) - p_0  \|_{\Besov{-s}{2}{1}(\cu_m)}^2 
\bigr]
+
\ahom^{-1}(\cu_0)
\E \bigl[
3^{-2tm} 
\| \a \nabla v(\cdot,\cu_m,p,q) - q_0 \|_{\Besov{-t}{2}{1}(\cu_m)}^2 
\bigr]
\\
\leq
C \delta^{\nf 12} (1-\gamma_1 - \gamma_2)^{-2} \Theta_0^{\nf 12}
\leq 
C  \delta^{\nf 12} \xi^2 \Theta_0^{\nf 12}
\,.
\end{multline}
In view of the previous two displays,~\eqref{e.Theta.m.bound} follows by selecting~$\ep := \delta^{\nf 14}(1-\gamma_1 - \gamma_2)^{-1}$ and using~\eqref{e.big.jump}.  Hence, after the claim~\eqref{e.weak.norms.in.the.proof} has been proven, we complete the proof. To this end, we only show the estimate in~\eqref{e.weak.norms.in.the.proof} for the first term on the left since the argument for the second one is completely analogous using~\eqref{e.weak.norms.flux.bound} instead of~\eqref{e.weak.norms.gradient.bound}. 

\smallskip

We estimate the terms on the right in~\eqref{e.weak.norms.gradient.bound} (applied with~$s' = \specs_1$). First, we claim that
\begin{align} 
\label{e.variance.term.one}
\lefteqn{
\E\biggl[ 
\biggl( 
\sum_{j=0}^m 3^{-s(m-j)}\biggl( \avsum_{z \in 3^j \Zd \cap \cu_m}
| (\a_*^{-1}(z+\cu_j) q - p ) - p_0  |^2 \biggr)^{\! \nf12}
\biggr)^{\!2}
\biggr]
} \quad &
\notag \\ &
\leq
C s^{-2} \ahom_*^{-1}(\cu_0) \Theta_0^{\nf 12}
\Bigl(  \delta^{\nf 12} + 3^{-\frac18 sm} \ahom_*(\cu_0) \E \bigl[ | \a_*^{-1}(\cu_0)|^2 \bigr]^{\nf 12} \Bigr)
\Bigl( \ahom_*^3(\cu_0) \E \bigl[ \lambda_{s,1}^{-3}(\cu_m)   \bigr] +  1\Bigr)^{\nf12}  
 \,.
\end{align}
By H\"older's inequality, we obtain that 
\begin{align*} 
\lefteqn{
\biggl( 
\sum_{j=0}^m 3^{-s(m-j)}\biggl( \avsum_{z \in 3^j \Zd \cap \cu_m}
\bigl| (\a_*^{-1}(z+\cu_j) q - p ) - p_0  \bigr|^2 \biggr)^{\! \nf 12}
\biggr)^{\! 2}
} \quad &
\notag \\ &
\leq
C s^{-2} \m_0 \bigl( \lambda_{s,1}^{-1}(\cu_m) +  \ahom_*^{-1}(\cu_m) \bigr)^{\! \nf 32} 
\biggl( 
s \sum_{j=0}^m 3^{-\frac14 s(m-j)}
\avsum_{z \in 3^j \Zd \cap \cu_m}
\bigl| \a_*^{-1}(z+\cu_j) - \ahom_*^{-1}(\cu_m) \bigr|\biggr)^{\! \nf 12}
\,.
\end{align*}
Thus, by taking the expectation and applying H\"older's inequality once more, we deduce that
\begin{align*} 
\lefteqn{
\E\biggl[ 
\biggl( 
\sum_{j=0}^m 3^{-s(m-j)}\biggl( \avsum_{z \in 3^j \Zd \cap \cu_m}
\bigl| (\a_*^{-1}(z+\cu_j) q - p ) - p_0  \bigr|^2 \biggr)^{\! \nf 12}
\biggr)^{\! 2}\biggr]
} \qquad &
\notag \\ &
\leq C s^{-2} \m_0 \E \Bigl[ \bigl( \lambda_{s,1}^{-1}(\cu_m) +  \ahom_*^{-1}(\cu_0) \bigr)^{3} \Bigr]^{\nf12} 
\biggl( s \sum_{j=0}^m 3^{-\frac14 s(m-j)}
\E\Bigl[
\bigl| \a_*^{-1}(\cu_j) - \ahom_*^{-1}(\cu_m) \bigr|\Bigr] \biggr)^{\! \nf 12}
\,.
\end{align*}
The last term can be estimated using stationarity, subadditivity~\eqref{e.subadditivity} and~\ref{a.frd} as
\begin{align*} 
\E\bigl[ | \a_*^{-1}(\cu_j) - \ahom_*^{-1}(\cu_m) |\bigr]
&
\leq
C \bigl( \ahom_*^{-1}(\cu_0) - \ahom_*^{-1}(\cu_m)\bigr)
+
\var\biggl[\avsum_{z \in \Zd \cap \cu_j} \a_*^{-1}(z+\cu_0) \biggr]^{\nf 12} 
\notag \\ & 
\leq
C \delta \ahom_*^{-1}(\cu_0) 
+ C 3^{-\frac d4 j} \E \bigl[ | \a_*^{-1}(\cu_0)|^2 \bigr]^{\nf 12}
\,.
\end{align*}
Combining the previous four displays yields~\eqref{e.variance.term.one}. Second, by using~\eqref{e.tau.mk.pq},~\eqref{e.J.moments} and~\eqref{e.lambda.moment.bound} together with H\"older's inequality and~\eqref{e.big.jump}, we get
\begin{align*} 
\lefteqn{
\E\biggl[
\lambda_{\specs_1,1}^{-1} (\cu_m) 
\biggl( 
\sum_{j=0}^m 
3^{-(s-\specs_1) (m-j)}
\biggl( \avsum_{z\in 3^j\Zd \cap \cu_m}
J(z+\cu_j,p,q) 
-
J(\cu_m,p,q)
\biggr)^{\!\nf12}\biggr)^{\!2} \biggr]
} \qquad\qquad &
\notag \\ &
\leq
(s-\specs_1)^{-2} \ahom_*^{-1}(\cu_0)
\Bigl(
\E\bigl[J(\cu_0,p,q) \bigr]
-
\E\bigl[ J(\cu_m,p,q)\bigr]\Bigr)
\notag \\ & \qquad 
+
(s-\specs_1)^{-2}  \E\Bigl[
\bigl(\lambda_{\specs_1,1}^{-1} (\cu_m) - \ahom_*^{-1}(\cu_0)\bigr)_+^2 
\Bigr]^{\nf 12} 
\E\bigl[ J(\cu_0,p,q)^2\bigr]^{\nf 12}
\notag \\ &
\leq
C (1-\specs_1-\specs_2)^{-2} 
\ahom_*^{-1}(\cu_0)\bigl( \delta \Theta_0^{\nf 12}
+ 3^{-(\specs_1 - \frac d\xi) m} \tilde{\Theta}_0 \bigr)
\leq
C \xi^2 
\ahom_*^{-1}(\cu_0)\delta  \Theta_0^{\nf 12}
\,.
\end{align*}
Third, by~\eqref{e.J.moments} and~\eqref{e.lambda.bounds.applied}, we obtain by~\eqref{e.big.jump} that
\begin{equation*} 
s^{-1} 3^{-2(s-\specs_1)m} 
\E\bigl[ \lambda_{\specs_1,1}^{-1} (\cu_m) 
J(\cu_m,p,q) \bigr] 
+
C s^{-2} 3^{-2sm} |p_0|^2 
\leq 
C \delta \ahom_*^{-1} (\cu_0)  \Theta_0^{\nf 12}
 \,.
\end{equation*}
Combining the above two displays with~\eqref{e.variance.term.one} and applying~\eqref{e.lambda.bounds.applied} once more together with~\eqref{e.big.jump} leads to the estimate for the first term on the left in~\eqref{e.weak.norms.in.the.proof}. As commented already before, the estimate for the second term is analogous. This concludes the proof. 
\end{proof}

Combining Lemmas~\ref{l.pigeon},~\ref{l.lambda.bounds} and~\ref{l.toss.J}, we obtain the following statement. 

\begin{lemma}
\label{l.one.step} 
There exists~$C(d) <\infty$ such that, for each~$\sigma \in (0,\nf12]$, 
\begin{equation*}
N \geq 
C\xi  
 \sigma^{-4} \log (\sigma^{-1}) 
\log \bigl(C \xi\tilde{\Theta}_0  \bigr)
\quad \implies \quad 
\tilde{\Theta}_N \leq 1 + \sigma \Theta_0 \,.
\end{equation*}
\end{lemma} 
\begin{proof} 
Let~$\delta \in (0,\nf12]$ to be selected below and let~$m \coloneqq  \bigl\lceil \xi \log \bigl( \delta^{-1}  \xi \tilde{\Theta}_0  \bigr) \bigr\rceil$. Assume that~$N \geq (k+1) m$ with~$k\coloneqq  \lceil 2 \delta^{-1}|\log\sigma| \rceil$. We intend to show that, if~$\delta \coloneqq c \sigma^4 \xi^{-4}$ with~$c(d)$ sufficiently small, then~$\Theta_{N-m} \leq 1+\frac12 \sigma \Theta_0$. 
We may assume that the first alternative of Lemma~\ref{l.pigeon} is valid, otherwise there is nothing to prove. We may therefore find~$n \in \{ m,\ldots,N-m\}$ such that 
\begin{equation}
\ahom(\cu_{n-m} ) \leq (1+\delta) \ahom(\cu_{n}) 
\quad \mbox{and} \quad 
\ahom_*^{-1} (\cu_{n-m} ) \leq (1+\delta) \ahom_*^{-1} (\cu_{n}) 
\,.
\end{equation}
Applying Lemma~\ref{l.toss.J} yields~$\Theta_n \leq 1 + C\xi \delta^{\nicefrac14} \Theta_{n-m}$. Taking~$c(d)>0$ sufficiently small and using the monotonicity of~$n\mapsto \Theta_n$, we obtain~$\Theta_{N-m} \leq 1 + \frac12 \sigma \Theta_0$, as claimed. 
By taking~$c$ smaller, if necessary, and applying~\eqref{e.tilde.Theta.vs.Theta} in Lemma~\ref{l.lambda.bounds}, we obtain  
\begin{equation*}
\tilde{\Theta}_{N}
\leq
(1+C\delta)\Theta_{N-m}
\leq 
1
+
\sigma \Theta_{0}\,.
\end{equation*}
This completes the proof. 
\end{proof} 

An iteration of Lemma~\ref{l.one.step} completes the proof of the main result. 

\begin{proof}[{Proof of Theorem~\ref{t.highcontrast}}] 
By performing~$C\log (1+\Theta_0)$ iterations of Lemma~\ref{l.one.step} with~$\sigma = \nf12$, we obtain
\begin{equation*}
N \geq C \log^2 (1+ \tilde{\Theta}_0)
\quad \implies \quad 
\tilde{\Theta}_N \leq 2 \,.
\end{equation*}
Applying it once more with~$\sigma \in (0,\nf 12]$ yields the theorem. 
\end{proof} 

\subsubsection*{\bf Acknowledgments}
S.A. was supported by NSF grant DMS-2350340. 
T.K. was supported by the Academy of Finland.

{\footnotesize
\bibliographystyle{alpha}
\bibliography{refs}
}

\end{document}